\documentclass{amsart}

\usepackage[all]{xy}
\usepackage{amsmath}
\usepackage{hyperref}
\usepackage{amsfonts,graphics,amsthm,amsfonts,amscd,latexsym}
\usepackage{epsfig}
\usepackage{flafter}
\usepackage{mathtools}
\usepackage{comment}
\usepackage{etoolbox}

\hypersetup{
    colorlinks=true,    
    linkcolor=blue,          
    citecolor=blue,      
    filecolor=blue,      
    urlcolor=blue           
}
\usepackage{tikz}
\usetikzlibrary{graphs,positioning,arrows,shapes.misc,decorations.pathmorphing}

\tikzstyle{vertex}=[
    draw,
    circle,
    fill=black,
    inner sep=1pt,
    minimum width=5pt,
]
\usepackage[position=top]{subfig}
\usepackage[alphabetic,backrefs]{amsrefs}
\usepackage{amssymb}
\usepackage{color}

\calclayout

\usetikzlibrary{decorations.pathmorphing}
\tikzstyle{printersafe}=[decoration={snake,amplitude=0pt}]

\newcommand{\rank}{\operatorname{rank}}

\newcommand{\supp}{\operatorname{supp}}

\newcommand{\pp}{\mathbb{P}}

\renewcommand{\qq}{\mathbb{Q}}
\newcommand{\zz}{\mathbb{Z}}
\newcommand{\nn}{\mathbb{N}}

\newcommand{\kk}{\mathbb{K}}

\def\O#1.{\mathcal {O}_{#1}}			
\def\pr #1.{\mathbb P^{#1}}				
\def\af #1.{\mathbb A^{#1}}			
\def\ses#1.#2.#3.{0\to #1\to #2\to #3 \to 0}	
\def\xrar#1.{\xrightarrow{#1}}			
\def\K#1.{K_{#1}}						
\def\bA#1.{\mathbf{A}_{#1}}			
\def\bM#1.{\mathbf{M}_{#1}}				
\def\bL#1.{\mathbf{L}_{#1}}				
\def\bB#1.{\mathbf{B}_{#1}}				
\def\bK#1.{\mathbf{K}_{#1}}			
\def\subs#1.{_{#1}}					
\def\sups#1.{^{#1}}

\usepackage{tikz}
\usetikzlibrary{matrix,arrows,decorations.pathmorphing}

\newtheorem{introdef}{Definition}

  \newtheorem{introthm}{Theorem}
  
  \newtheorem{introrem}{Remark}
  \newtheorem{introcor}{Corollary}

  \newtheorem{theorem}{Theorem}[section]
  \newtheorem{lemma}[theorem]{Lemma}
  \newtheorem{proposition}[theorem]{Proposition}
  \newtheorem{corollary}[theorem]{Corollary}

  \newtheorem{definition}[theorem]{Definition}
  \newtheorem{example}[theorem]{Example}

\newtheorem{remark}[theorem]{Remark}

\theoremstyle{remark}

\numberwithin{equation}{section}

\begin{document}

\title[Kawamata log terminal singularities of full rank]{Kawamata log terminal singularities of full rank}

\author[J.~Moraga]{Joaqu\'in Moraga}
\address{Department of Mathematics, Princeton University, Fine Hall, Washington Road, Princeton, NJ 08544-1000, USA
}
\email{jmoraga@princeton.edu}

\maketitle

\begin{abstract}
We study Kawamata log terminal singularities of full rank, i.e., 
$n$-dimensional klt singularities containing a large finite abelian group of rank $n$ in its regional fundamental group.
The main result of this article 
is that klt singularities of full rank degenerate to cones over log crepant equivalent
toric quotient varieties.
To establish the main theorem, we reduce the proof to the study of Fano type varieties with large finite automorphisms of full rank.
We prove that such Fano type varieties are log crepant equivalent  toric. 
Furthermore, any such Fano variety of dimension $n$ contains an open affine subset isomorphic to $\mathbb{G}_m^n$.
As a first application,
we study complements on klt singularities of full rank.
As a second application, we study dual complexes of log Calabi-Yau structures on
Fano type varieties with large fundamental group of their smooth locus.
\end{abstract}

\setcounter{tocdepth}{1} 
\tableofcontents

\section{Introduction}

The study of the singularities of the minimal model program
is a cornerstone to prove new results about projective varieties.
Many of the important conjectures of the minimal model program
can be reduced to conjectures about Kawamata log terminal singularities.
Henceforth, a better understanding of klt singularities is a central topic on birational geometry.
On the other hand, via global-to-local methods, we can deduce theorems about
these singularities, 
by reducing it to the study of Fano type varieties.
Recently, the boundedness of Fano varieties and the theory of complements~\cites{Bir19,Bir21}
had lead to new results about klt singularities.
First, we have the existence of bounded log canonical complements
on klt singularities~\cite{Bir19}.
Secondly, we have the boundedness up to degeneration of exceptional klt singularities~\cites{Mor18b,HLM20} and the ascending chain condition for minimal log discrepancies
of exceptional singularities~\cites{Mor18a,HLS19}.
Finally, we have the Jordan property for the regional fundamental group of $n$-dimensional klt singularities~\cite{BFMS20}.
This result relies on work by Prokhorov and Shramov
on the Jordan property for Cremona groups~\cites{PS14,PS16}.
The Jordan property asserts that the local fundamental group of $n$-dimensional klt singularities are almost abelian of rank at most $n$. 
In many cases, for instance, surface singularities and toric singularities, 
much of the local geometry of the singularity can be deduced from its
local fundamental group.

However, in general, not much can be said unless the fundamental group of the klt singularity is large enough.
Then, it is natural to ask whether the existence of a large abelian group on the local fundamental group will reflect on the geometry of the singularity.
The first step in this direction was given in~\cite{Mor20}, where this question
was settled in dimension three, through the study of 
large finite automorphisms on Fano type surfaces.
The author introduces the class of log crepant equivalent  toric quotient singularities (abbreviated lce-tq), which is a natural class of singularities including
toric singularities, quotient singularities, and certain crepant equivalent  models of those.
These singularities are cones over log crepant equivalent  toric quotient projective varieties (see Definition~\ref{introdef}).
\cite{Mor20}*{Theorem 5.1} asserts that a klt $3$-fold singularity
$x\in (X,\Delta)$ which satisfies $\zz_N^3\leqslant \pi_1^{\rm loc}(X,\Delta;x)$
for $N$ large enough (compared with the dimension) will admit a degeneration
to a lce-tq singularity.
Log crepant equivalent  toric quotient singularities are a prototype of singularities with large local fundamental groups.

In this article, we study Kawamata log terminal singularities of full rank, i.e., $n$-dimensional klt singularities containing a large finite abelian group of rank $n$ in its regional fundamental group.
Our main theorem is  a characterization of these singularities.
We generalize~\cite{Mor20}*{Theorem 5.1} to arbitrary dimension.
We prove that Kawamata log terminal singularities of full rank
degenerate to log crepant equivalent  toric quotient singularities.

\begin{introthm}\label{introthm:degeneration}
Let $n$ be a positive integer.
There exists a constant $N=N(n)$, only depending on $n$, satisfying the following.
Let $x\in (X,\Delta)$ be a $n$-dimensional klt singularity such that
$\zz^n_k \leqslant \pi_1^{\rm reg}(X,\Delta;x)$ for some $k\geq N$.
Assume that $\Delta$ has standard coefficients.
Then, $x\in (X,\Delta)$ degenerate to a log crepant equivalent  toric quotient singularity.
\end{introthm}

The above theorem says that among all $n$-dimensional klt singularities, 
the ones that have the largest local fundamental groups, 
belong to the versal deformation space of log crepant equivalent  toric quotient singularities of dimension $n$.
Recall that klt surface singularities are quotient singularities (see, e.g.~\cite{Tsu83}). Hence, the above theorem is vacuous in dimension two.
It is only non-trivial starting in dimension three.
In Example~\ref{example}, we give an example of a log crepant equivalent  toric quotient singularity which is not the quotient of a toric singularity.
All the results stated in the introduction still hold for subgroups that are not of the form $\zz^n_k$.
However, to state the theorems properly, we need the concept of large $n$-generation of finite groups (see subsection~\ref{subsec:k-gen}).
In Theorem~\ref{thm:regional-general}, 
we give a more general version of Theorem~\ref{introthm:degeneration} admitting more general coefficients. The above theorem can be reduced to the study of large finite abelian groups of rank $n$ in Fano type varieties.
Theorem~\ref{introthm:degeneration} will be deduced from Theorem~\ref{introthm:main-thm} below.
This is an application of purely log terminal blow-ups to reduced problems about klt singularities to problems about Fano type varieties (see, e.g.,~\cites{Pro00,Xu14}).

\begin{introrem}{\em 
We expect similar results when the regional fundamental group of a klt singularity
is large but possibly it does not have full rank.
To tackle that direction, it seems necessary to improve the main result of~\cite{BMSZ18} to the relative setting.
For $3$-fold singularities, the cases of rank two and rank three fundamental groups are already considered in~\cite{Mor20}.
In low dimensions, a relative version of~\cite{BMSZ18} is already proved in~\cite{Sho00}. 
When we consider klt singularities with cyclic regional fundamental groups,
the techniques of these articles seem not to apply to deduce geometry of the singularity.
}
\end{introrem}

\subsection{Fano type varieties with full rank automorphisms} In this subsection, we introduce the results regarding Fano type varieties with large full rank automorphism groups.
We start by stating the main projective theorem of this article.

\begin{introthm}\label{introthm:main-thm}
Let $n$ be a positive integer.
There exists a positive integer
$N:=N(n)$, only depending on $n$, satisfying the following.
Let $X$ be a Fano type variety of dimension $n$.
Let $A\simeq \zz_k^n \leqslant {\rm Aut}(X)$ be a finite subgroup with $k\geq N$.
Then, there exists:
\begin{enumerate}
\item a boundary $B$ on $X$, and
\item an $A$-equivariant birational map
$X\dashrightarrow X'$,
\end{enumerate}
satisfying the following conditions:
\begin{enumerate}
\item The pair $(X,B)$ is log canonical, $A$-equivariant, and $K_X+B\sim 0$,
\item the push-forward of $K_X+B$ to $X'$ is a log pair $(X',B')$, 
\item the pair $(X',B')$ is a log Calabi-Yau toric pair, and
\item there are group monomorphisms 
$A< \mathbb{G}_m^n\leqslant {\rm Aut}(X,B)$.
\end{enumerate}
In particular, $B'$ is the reduced toric boundary of $X'$.
Furthermore, the birational map
$X'\dashrightarrow X$ is an isomorphism over the torus $\mathbb{G}_m^n$.
\end{introthm}

The main ingredients of the proof of Theorem~\ref{introthm:main-thm} are: the existence of $G$-equivariant $M$-complements on Fano type varieties~\cites{Bir19,Mor20}, 
the boundedness of Fano varieties with bounded singularities~\cite{Bir21},
some basic tools from the theory of reductive groups~\cites{Bor91}, and
the characterization of projective toric varieties using complexity~\cite{BMSZ18}.
In Theorem~\ref{thm:FT-full-rank}, we give a more general version that deals with pairs.
In Theorem~\ref{thm:FT-full-rank-non-abelian-G}, we give a version that deals with non-abelian groups.
We introduce some notation to state a version of Theorem~\ref{introthm:main-thm} which may be more natural to the reader.

\begin{introdef}\label{introdef}
{\em We say that a log canonical pair $(X,B)$ is {\em crepant equivalent toric} if $(X,B)$ is crepant equivalent  to a projective toric pair.
In particular, a $n$-dimensional  crepant equivalent  log toric pair $(X,B)$ is crepant equivalent  to
$(\pp^n, H_1+ \dots +H_{n+1})$ 
(see Lemma~\ref{lem:toric-bir-pn}).
We say that a projective variety $X$ is {\em log crepant equivalent  toric} if $(X,B)$ is crepant equivalent  toric for some boundary $B$ on $X$.
We say that a projective variety $X$ is {\em log crepant equivalent  toric quotient} if it is the finite quotient of a log crepant equivalent  toric variety. 
A klt singularity is called {\em log crepant equivalent  toric quotient singularity} if it is the cone over
a log crepant equivalent  toric quotient projective variety.
We may write {\em lce-tq} to abbreviate {\em log crepant equivalent toric quotient}.
}
\end{introdef}

\begin{introcor}\label{introcor-easy-version}
A Fano type variety of dimension $n$ with a large
abelian automorphism group of rank $n$ is crepant equivalent  log toric.
\end{introcor}

Here, {\em large} means that $A\simeq \zz^n_k$, where $k$ is larger than a universal constant
which only depends on $n$, as in the statement of Theorem~\ref{introthm:main-thm}.
Some interesting aspect of the main theorem is that it implies the existence of an open set of $X$ isomorphic to an algebraic torus.
Hence, the projective variety $X$ is a possibly non-toric Fano type compactification of the algebraic torus so that the action of a large discrete subgroup of the torus extends to the whole variety.
On this open subset $A$ acts as the multiplication by roots of unity.
Furthermore, there is a $1$-complement of the Fano type variety which is supported on the complement of the algebraic torus.
We make this precise in the next corollary.

\begin{introcor}\label{introcor1}
Let $n$ be a positive integer.
There exists a positive integer $N:=N(n)$, only depending on $n$, satisfying the following.
Let $X$ be a Fano type variety of dimension $n$ and $A\simeq \zz_k^n \leqslant {\rm Aut}(X)$ be a finite subgroup with $k\geq N$.
Then, there exists:
\begin{enumerate}
\item An open subset $U\subset X$ which is $A$-invariant, and
\item a reduced boundary $B\subset X\setminus U$,
\end{enumerate} 
satisfying the following:
\begin{enumerate} 
\item $(X,B)$ is log canonical,
$A$-equivariant, and $K_X+B\sim 0$,
\item $U$ is isomorphic to $\mathbb{G}_m^n$, and
\item $A$ acts on $U$ by multiplication of roots of unity under the
above isomorphism.
\end{enumerate} 
\end{introcor}

In general, the boundary $B$ can be strictly supported on the complement of the torus on $X$.
The above results imply that large discrete algebraic dynamics on a Fano type variety become continuous on an affine open subset.

\subsection{Log smooth locus of Fano type varieties}
In this subsection, we state results regarding Fano type variety
with $\pi_1^{\rm alg}(X^{\rm sm})$ being a large group.
Here, $X^{\rm sm}$ is the smooth locus of the variety $X$.
Toric projective varieties can be examples of such Fano type varieties.
The first result, says that such Fano type varieties are log crepant equivalent  toric quotients in the sense of Definition~\ref{introdef}.

\begin{introthm}\label{introthm:log-smooth-locus}
Let $n$ be a positive integer.
There exists a constant $N:=N(n)$, only depending on $n$,
satisfying the following.
Let $X$ be a Fano type variety of dimension $n$
so that $\zz_k^n \leqslant \pi_1^{\rm alg}(X^{\rm sm})$
with $k\geq N$.
Then $X$ is a log crepant equivalent  toric quotient projective variety.
\end{introthm}

The second result in this direction is that we can find a
boundary $B$ on these varieties so that $(X,B)$ is log canonical,
$K_X+B\sim_\qq 0$, and the dual complex $\mathcal{D}(X,B)$
is PL-homeomorphic to the quotient of a sphere.

\begin{introthm}\label{introthm:dual-complex}
Let $n$ be a positive integer.
There exists a positive integer $N:=N(n)$, only depending on $n$,
satisfying the following.
Let $X$ be a Fano type variety of dimension $n$ so that
$\zz^n_k \leqslant \pi_1^{\rm alg}(X^{\rm sm})$ with $k\geq N$.
Then, there exists a boundary $B$ on $X$ so that
$\mathcal{D}(X,B)\simeq_{\rm PL} S^{n-1}/G$
where $G$ is a finite group
with $|G|\leq N$.
\end{introthm}

Observe that our result coincides with the expectation  in~\cite{KX16}*{Question 4}.
However, we do not claim that this holds for every boundary on $X$ (see, e.g.,~\cite{KX16}*{Example 60}).
On the other hand, in our result, we can also control the order of the group.
As usual, the two above results can be stated for more general groups and boundaries. This is done in Section~\ref{sec:app}.

\subsection{Complements on klt singularities of full rank}
Our last result concerns bounded covers of klt singularities of full rank.
We prove that, given a klt singularity of full rank, 
we can find a finite Galois cover with degree bounded by a constant on the dimension so that such cover admits a $1$-complement.

\begin{introthm}\label{introthm:complements}
Let $n$ and $m$ be positive integers.
There exists a positive integer $N:=N(n,m)$, only depending on $n$ and $m$, satisfying the following.
Let $x\in (X,\Delta)$ be a $n$-dimensional klt singularity
such that $\zz^n_k \leqslant \pi_1^{\rm reg}(X,\Delta;x)$
for some $k\geq N$.
Assume that $m\Delta$ is a Weil divisor.
Then, there exists a finite Galois cover $\pi\colon X'\rightarrow X$
of index at most $N$, 
so that the log pull-back $K_{X'}+\Delta_{X'}=\pi^*(K_X+\Delta)$
admits a $1$-complement.
\end{introthm}

Note that the above is a classic and important property of quotient singularities.
Let $X=\mathbb{A}_\mathbb{K}^n/G$ be a quotient singularity
where $G<{\rm GL}_n(\mathbb{K})$ is a finite group acting without ramification in codimension one.
Let $x\in X$ be the image of the origin.
By the classic Jordan property~\cite{Jor73}, 
we can find a normal abelian subgroup $A\leqslant G<{\rm GL}_n(\mathbb{K})$ of index $N(n)$ on $G$.
Since $A$ is an abelian group acting on $\mathbb{A}_\mathbb{K}^n$,
we can find a diagonalization of this action.
Hence, the quotient $T=\mathbb{A}^n_\mathbb{K}/A$ is a $\qq$-factorial toric singularity.
The toric singularity $t\in T$ admits a quotient $T\rightarrow X$ whose degree is bounded by $N(n)$, i.e., quotienting by the \textit{non-abelian} part of $G$.
Then, we conclude that $X$ admits a bounded Galois cover
$T$ which admits a $1$-complement; the toric complement.
In general, it is not true that $x\in X$ admits a $1$-complement.
Recall that, we know that $x\in X$ admits a complement which only depends on the dimension (see, e.g.,~\cite{Bir19}).

\subsection*{Acknowledgements} 
The author would like to thanks
Stefano Filipazzi,
J\'anos Koll\'ar, 
Mirko Mauri,
Constantin Shramov, 
and Burt Totaro 
for many useful comments.

\section{Sketch of the Proof}
In this section, we give a brief sketch of the proof of our main projective theorem. The proof aims to find a boundary $B$ on $X$ and
an equivariant birational map 
$X\dashrightarrow X^*$ so that the induced log pair $(X^*,B^*)$ has low complexity. Then, we will use the techniques of subsection~\ref{subsec:complexity} to deduce that $(X^*,B^*)$ is a toric pair.
The fact that $X\dashrightarrow X^*$ is an isomorphism over the torus will be a consequence of the construction of such birational map.
We proceed to sketch how to find such boundary $B$ and such birational map.

We start with a Fano type variety $X$ with a finite group action $A$ which has large $n$-generation.
We assume that $A$ is abelian.
Essentially, $A$ will contain $\zz^n_N$ with $N$ large,
provided that $g_n(A)$ is large enough (see subsection~\ref{subsec:k-gen}).
Throughout the sketch, we may assume $A\simeq \zz^n_N$.
Then, using the theory of $G$-equivariant complements (see subsection~\ref{subsec:g-equiv-comp}),
we may find a boundary $B$ on $X$ which is $A$-invariant.
We can find $B\geq 0$ so
that $(X,B)$ has log canonical singularities and
$M(K_X+B)\sim 0$. Here, $M$ only depends on the dimension of $X$. 
The underlying idea is that if we fix $M(K_X+B)\sim 0$
and let $A\simeq \zz^n_N$ be large enough, then the boundary $B$ must be reduced. So the complexity of $(X,B)$ will be small
if we can control the Picard rank of $X$.
However, at this point, we can't control the Picard rank of $X$. 
To solve this issue,
we must run a minimal model program to solve this issue.

We sight to prove the statement either by reducing to a bounded family or inductively.
We replace $X$ by an $A$-equivariant canonicalization.
This previous step is done using the techniques of the equivariant minimal model program (see subsection~\ref{subsec:g-equiv-mmp}).
We run an $A$-equivariant minimal model program for $K_X$
which we are assuming to have canonical singularities.
If it terminates with an $A$-equivariant Mori fiber space to a point $X'\rightarrow {\rm Spec}(\kk)$,
then we will deduce that such varieties belong to a bounded family. Indeed, we will have that $-K_{X'}$ is ample and $X'$ has canonical singularities. In this case, replacing $A$ with a bounded subgroup (which only depends on this bounded family), we may assume that
$A$ lies in the connected component ${\rm Aut}(X',B')^0$
of ${\rm Aut}(X',B')$. 
Thus, we may use techniques from connected linear algebraic groups to deduce that $A$ must lie in a maximal torus of
${\rm Aut}(X',B')^0$.
Hence, such torus must have the same dimension as $X'$.
We conclude that
$X'$ has a torus action and the induced boundary $B'$ is toric (see $\S$~\ref{sec:bounded}).
Then, using the techniques developed in section~\ref{subsec:toric-geometry}, we will prove that
indeed $A<\mathbb{G}_m^n\leqslant {\rm Aut}(X',B')$.

Then, we reduce to the case in which have an $A$-equivariant Mori fiber space $X'\rightarrow C$.
We will have an exact sequence $1\rightarrow A_F\rightarrow A\rightarrow A_C\rightarrow 1$ where $A_F$ acts fiber-wise and $A_C$ acts on the base.
We let $f$ be the dimension of $F$ and $c$ be the dimension of $C$.
Then, we prove that $A_F$ (resp. $A_C$) contains $\zz_N^f$ (resp. $\zz_N^c)$, for some $N$ large enough.
Hence, we can apply the inductive hypothesis to a general fiber and the base.
By the inductive hypothesis, we may find a $A_C$-equivariant birational map
$\pp^c\dashrightarrow C$.
By considering an $A$-equivariant log resolution of $(X',B')$ and running a suitable $A$-equivariant minimal model program over $\pp^c$, we will end up with $A$-equivariant fibration $X^*\rightarrow \pp^c$.
In this part of the proof, we will need to control the divisors contracted by the minimal model program to argue that $X^*$ is still of Fano type.
Hence, we will use the theory of degenerate divisors (see $\S$~\ref{sec:deg}).
We may re-compactify the complement of the pre-image of $\mathbb{G}_m^c$ on $X^*$ in such a way that $X^*$ has $A$-invariant Picard rank two.
Moreover, we prove that the morphism $X^*\rightarrow \pp^c$ is an $A$-equivariant Mori fiber space.
We aim to prove that $X^*$ is indeed a toric variety
and the induced boundary $B^*$ is log toric.

We would like to control the Picard rank of $X^*$.
However, we can only control its $A$-invariant Picard rank.
So the next step will be quotienting by $A$ to obtain a
Mori fiber space $Y\rightarrow T_Y$,
where $T_Y$ is a toric quotient of $\pp^c$.
In particular, $T_Y$ has Picard rank one and $Y$ has Picard rank two.
The problem with the model $Y$ 
is that even if we control the general log fiber, which is log toric, we don't control the number of horizontal components of the induced boundary $B_Y$.
So, the next step is to take a sequence of Galois covers terminating on a contraction $Z\rightarrow T_Z$ 
so that the horizontal divisors of $B_Z$ are unramified over the base.
We control the degree of the cover $Z\rightarrow Y$ by controlling the Picard rank of the general fiber of $Y\rightarrow T_Y$.
We will construct $Z$ in such a way that $T_Z$ is toric of Picard rank one.
However, since we took a Galois cover, we may lose control of the Picard rank of $Z$.
Hence, we will run a further minimal model program from a log resolution of $Z$ over $T_Z$ which terminates in a Mori fiber space $Z'\rightarrow T_Z$.
In particular, $Z'$ will have both; Picard rank two and enough components on the induced boundary $\lfloor B_{Z'}\rfloor$. Then, we can use the complexity to deduce that
$(Z',B_{Z'})$ is toric.
Then, we argue that $Z\rightarrow Z'$ only extract log canonical places of $(Z',B_{Z'})$ so $(Z,B_Z)$ is toric as well.

Then, we define $W$ to be the normalization of a connected component of $Z\times_Y X^*$ and prove that is a toric cover of $Z$.
Thus, we have finite Galois morphism $W\rightarrow X$ from a projective toric variety $W$.
Furthermore, we control the degree of $W\rightarrow X$ by controlling the Picard rank of the general fiber of $X^*\rightarrow \pp^c$.
Finally, we will argue that $X^*$ is the quotient of $W$ by a finite group $H$, with order $O_n(1)$,
which is normalized by a large discrete subgroup of the torus of $W$.
If such a discrete group is large enough, then $H$ must be a subgroup of the torus as well (see subsection~\ref{subsec:toric-geometry}).
We conclude that $X^*$ is the quotient of a projective toric variety by a finite subgroup of the torus, hence $X^*$ is itself toric. 

\section{Preliminaries}

In this section, we recall the preliminaries that will be used in this article:
the singularities of the minimal model program,
Fano type varieties, $G$-invariant complements,
and toric geometry. 
Throughout this article, we work over an algebraically closed field $\mathbb{K}$ of characteristic zero.
We denote by $\mathbb{G}_m^k$ the $k$-dimensional $\mathbb{K}$-torus.
Galois morphisms are required to be finite and surjective.
However, we do not assume that Galois morphisms are \'etale.
The rank of a finite group refers to the minimal number of generators.
We say that a quantity is $O_n(1)$ if it is bounded by $C_n$, 
where $C_n$ is a quantity depending only on $n$.
We will use some standard results on the minimal model program~\cites{KM98,BCHM10} and toric geometry~\cites{Ful93,CLS11,Cox95}.

\subsection{Kawamata log terminal singularities}\label{subsec:klt}
In this subsection, we recall definitions regarding the singularities of the minimal model program.

\begin{definition}
{\em
A {\em contraction} is a morphism $f\colon Y\rightarrow X$
so that $f_*\mathcal{O}_Y=\mathcal{O}_X$.
A {\em fibration} is a contraction with positive dimensional general fiber.
}
\end{definition}

\begin{definition}
{\em A {\em log pair} or a {\em pair} is a couple
$(X,\Delta)$ where $X$ is a normal algebraic variety and $\Delta$ is an effective $\qq$-divisor on $X$ so that $K_X+\Delta$ is a $\qq$-Cartier $\qq$-divisor.
}
\end{definition} 

\begin{definition}{\em 
Let $X$ be a projective variety.
A {\em prime divisor over $X$} is a prime divisor which is contained in a variety which admits a projective birational morphism to $X$.
This means that there is a projective birational morphism
$\pi\colon Y\rightarrow X$ so that $E\subset Y$
is a prime divisor.

Let $(X,\Delta)$ be a log pair and $E$ be a prime divisor over $X$.
We define the {\em log discrepancy}
of $(X,\Delta)$ at $E$ to be
\[
a_E(X,\Delta):=
{\rm coeff}_E(K_Y-\pi^*(K_X+\Delta))
\]
where as usual we pick $K_Y$ so that $\pi_*K_Y=K_X$.
A {\em log resolution} of a log pair $(X,\Delta)$ is a projective birational morphism $\pi\colon Y\rightarrow X$,
with purely divisorial exceptional locus,
so that
$Y$ is a smooth variety and $\pi^{-1}_*\Delta+{\rm Ex}(\pi)_{\rm red}$ is a divisor with simple normal crossing.
By Hironaka's resolution of singularities, we know that any
log pair admits a log resolution.
}
\end{definition}

The log discrepancies of a log pair allow us to measure the singularities of the pair.

\begin{definition}{\em 
A pair $(X,\Delta)$ is said to be {\em Kawamata log terminal} (or {\em klt} for short) if all its log discrepancies are positive, 
i.e., $a_E(X,\Delta)>0$ for every prime divisor $E$ over $X$. It is known that a pair $(X,\Delta)$ is klt if and only if all the log discrepancies corresponding to prime divisors on a log resolution of $(X,\Delta)$ are positive.

A pair $(X,\Delta)$ is said to be {\em log canonical} (or {\em lc} for short) if all its log discrepancies are non-negative, 
i.e., $a_E(X,\Delta)\geq 0$ for every prime divisor $E$ over $X$.
As in the case of klt singularities, a pair $(X,\Delta)$ is log canonical if and only if all the log discrepancies corresponding to prime divisors on a log resolution of $(X,\Delta)$ are non-negative (see, e.g.,~\cite{KM98}).

A pair $(X,\Delta)$ is said to be {\em $\epsilon$-log canonical}
(resp. {\em $\epsilon$-klt}) if $a_E(X,\Delta)\geq \epsilon$
(resp. $a_E(X,\Delta)>\epsilon$) for every prime divisor $E$ over $X$.
}
\end{definition}

\begin{definition}{\em 
Let $(X,\Delta)$ be a log pair.
A {\em log canonical place} (resp. {\em non-klt place}) of $(X,\Delta)$ is a prime divisor $E$ over $X$ so that $a_E(X,\Delta)=0$ (resp. $a_E(X,\Delta)\leq 0$).
A {\em log canonical center} of $(X,\Delta)$ (resp.{\em non-klt center}) is the image on $X$ of a log canonical place of $(X,\Delta)$
(resp. of a non-klt place of $(X,\Delta)$).
}
\end{definition}

\begin{definition}{\em 
A pair $(X,\Delta)$ is said to be {\em divisorially log terminal},
or {\em dlt} for short, if there exists an open set $U\subset X$ 
satisfying the following:
\begin{enumerate}
\item $U$ is smooth and $\Delta|_U$ has simple normal crossing, and
\item the coefficients of $\Delta$ are less or equal than one,
\item all the non-klt centers of $(X,\Delta)$ intersect $U$ and
are given by strata of $\lfloor \Delta\rfloor$.
\end{enumerate}
A pair $(X,\Delta)$ is said to be {\em purely log terminal},
or {\em plt} for short, if it is dlt and it has at most one log canonical center.
In particular, plt singularities are dlt, 
and klt singularities are both plt and dlt
}
\end{definition}

\begin{definition}{\em 
Let $(X,\Delta)$ be a klt singularity and $x\in X$ be a closed point.
A {\em purely log terminal blow-up} of $(X,\Delta)$ at $x$
(or {\em plt blow-up} for short)
is a projective birational morphism $\pi\colon Y\rightarrow X$ satisfying the following:
\begin{enumerate}
\item $\pi$ is an isomorphism over $X\setminus \{x\}$ and the pre-image of $x\in X$ is a prime divisor $E\subset Y$, 
\item the pair $(Y,\Delta_Y+E)$ is plt, where $\Delta_Y$ is the strict transform of $\Delta$, and
\item $-E$ is ample over $X$.
\end{enumerate}
The existence of plt blow-ups for klt singularities is known (see, e.g.,~\cites{Pro00,Xu14}).
}
\end{definition}

\begin{definition} 
{\em
Let $(X,\Delta)$ be a log pair.
We define ${\rm Aut}(X,\Delta)$ to be the subgroup 
of ${\rm Aut}(X)$ given by automorphisms $g\in {\rm Aut}(X)$ 
so that $g^* \Delta=\Delta$.
In particular, every such $g$ must map prime divisors of
$\Delta$ to prime divisors of $\Delta$ with the same coefficients.
If $(X,\Delta)$ is a klt (resp. lc) pair 
and $G\leqslant {\rm Aut}(X,\Delta)$, then the quotient $(Y,\Delta_Y)$
of $(X,\Delta)$ by $G$ is again klt (resp. lc). 
See, for instance~\cite{Sho93}*{\S 2}.
In the above case, we say that $(X,\Delta)$ is a {\em $G$-invariant klt singularity}.
}
\end{definition}

\begin{definition}
Let $(X,\Delta)$ be a $G$-invariant klt pair and $x\in X$ be a closed point which is fixed with respect to the $G$-action.
A {\em $G$-invariant plt blow-up} of $(X,\Delta)$ at $x$ 
is a $G$-invariant projective birational morphism $\pi\colon Y\rightarrow X$ satisfying the following:
\begin{enumerate}
\item $\pi$ is an isomorphism over $X\setminus \{x\}$ and the pre-image of $x\in X$ is a prime divisor $E\subset Y$, 
\item the pair $(Y,\Delta_Y+E)$ is $G$-plt, where $\Delta_Y$ is the strict transform of $\Delta$, and
\item $-E$ is ample over $X$.
\end{enumerate}
\end{definition}

To conclude this subsection, we prove the existence
of $G$-invariant plt blow-ups for $G$-invariant klt singularities.

\begin{lemma}\label{lem:existence-g-inv-plt}
Let $(X,\Delta)$ be a $G$-invariant klt pair 
and $x\in X$ be a closed $G$-fixed point.
Then, there exists a $G$-invariant plt blow-up for $(X,\Delta)$ at $x$.
\end{lemma} 

\begin{proof}
The quotient $(X',\Delta')$ of $(X,\Delta)$ with respect to $G$ is again a klt pair. Let $x'\in X'$ be the image of $x$ on $X'$.
Let $\pi\colon Y'\rightarrow X'$ be a plt blow-up of $(X',\Delta')$ at $x'$.
Let $Y$ be the normalization of the main component of
$Y'\times_{X'} X$.
We have a projective birational morphism 
$Y\rightarrow X$.
We denote by $K_Y+\Delta_Y+E_Y$ the log pull-back
of $K_{Y'}+\Delta_{Y'}+E_{Y'}$.
Note that $K_Y+\Delta_Y+E_Y$ is anti-ample over the base.
Hence, we conclude that $-E_Y$ is ample over $X$.
The morphism $Y\rightarrow X$ is an isomorphism over
$X\setminus \{x\}$ by construction.
We claim that $(Y,\Delta_Y+E_Y)$ is plt.
The pair $(Y,\Delta_Y+E_Y)$ is $G$-invariant log canonical by construction, where $\Delta_Y$ is the strict transform
of $\Delta$ on $Y$.
Since $-(K_Y+\Delta_Y+E_Y)$ is ample over the base,
the connectedness principle says that the non-klt locus
of such pair is connected over the base.
On the other hand, all log canonical centers of $(Y,\Delta_Y+E_Y)$ are divisorial,
since all log canonical centers of $(Y',\Delta_{Y'}+E_{Y'})$
are divisorial.
We conclude that all the log canonical centers
of $(Y,\Delta_Y+E_Y)$ must be contained in
$\lfloor E_Y\rfloor$
and this set is connected over the base.
Thus, if $E_Y$ has more than one component, then
$(Y,\Delta_Y+E_Y)$ has at least one lcc of codimension at least two. Leading to a contradiction.
We conclude that $E_Y=\lfloor E_Y\rfloor$ is prime
and $G$-invariant.
Hence, $(Y,\Delta_Y+E_Y)$ is a $G$-equivaraint plt pair.
Thus, $\pi\colon Y\rightarrow X$ is a $G$-equivariant plt blow-up. This concludes the proof.
\end{proof}

\subsection{Fano type varieties}\label{subsec:ft-var}
In this subsection, we recall the definition of Fano type varieties and the boundedness results about these varieties.

\begin{definition}{\em 
Let $X$ be a projective variety.
We say that $X$ is a {\em Fano type variety},
if there exists a boundary $\Delta$ on $X$ 
so that $(X,\Delta)$ has klt singularities
and $-(K_X+\Delta)$ is a big and nef divisor.
We say that $X$ is a {\em Fano variety}
if $X$ has klt singularities
and $-K_X$ is an ample $\qq$-Cartier divisor.
}
\end{definition}

It is well-known that Fano type varieties
are Mori dream spaces~\cite{BCHM10}*{Corollary 1.3.2}.
In particular, for any $\qq$-divisor $D$ on $X$
we may run a minimal model program 
which either terminates with a Mori fiber space
or a good minimal model~\cite{HK00}*{Proposition 1.11}.

The following proposition is well-known, we will use it often in this article.

\begin{proposition}
A projective variety $X$ is of Fano type if and only if 
there exists a boundary $\Delta$ on $X$ so that $K_X+\Delta\sim_\qq 0$,
$(X,\Delta)$ has klt singularities, and
$\Delta$ is a big divisor on $X$.
\end{proposition}

The following theorem, due to Birkar, is known as the boundedness of Fano type varieties.

\begin{theorem}[cf.~\cite{Bir21}*{Theorem 1.1}]
Let $n$ be a positive integer and let $\epsilon$ be a positive real number.
Then, the projective varieties $X$ such that
\begin{enumerate}
\item $(X,\Delta)$ is $\epsilon$-lc and $n$-dimensional for some boundary $\Delta$, and  
\item  $-(K_X+\Delta)$ is big and nef,
\end{enumerate}
form a bounded family.
\end{theorem} 

The above theorem is often written in the following equivalent form.

\begin{theorem} 
Let $n$ be a positive integer and let $\epsilon$ be a positive real number.
The set of $n$-dimensional projective varieties $X$ so that there exists a boundary $\Delta$ on $X$ satisfying:
\begin{enumerate}
\item $K_X+\Delta\sim_\qq 0$, 
\item $(X,\Delta)$ has $\epsilon$-log canonical singularities, and
\item $\Delta$ is a big divisor on $X$,
\end{enumerate}
form a bounded family.
\end{theorem} 

The following is a log version of the above theorem. For a proof see~\cite{FM20}*{Theorem 2.9}.

\begin{theorem}\label{thm:log-ft-bdness}
Let $n$ be a positive integer,
$\epsilon$ be a positive real number,
and $\mathcal{R}$ be a finite set of rational numbers.
The set set of $n$-dimensional projective pairs $(X,\Delta)$ so that
\begin{enumerate}
    \item $(X,\Delta)$ is $\epsilon$-log canonical,
    \item $K_X+\Delta\sim_\qq 0$, 
    \item $\Delta$ is a big divisor on $X$, and
    \item the coefficients of $\Delta$ are in $\mathcal{R}$,
\end{enumerate}
form a log bounded family.
\end{theorem}

The following proposition is well known, it allows us to extract non-canonical centers from a Fano type variety and stays in the same class of varieties.

\begin{proposition}\label{prop:FT-extraction-non-canonical}
Let $X$ be a Fano type variety.
Let $B$ be an effective divisor on $X$ so that
$(X,B)$ has log canonical singularities and
$K_X+B\sim_\qq 0$.
Let $Y\dashrightarrow X$ be a birational map which only extract prime divisors with log discrepancy in the interval $[0,1)$ with respect to $(X,B)$.
Then, the variety $Y$ is Fano type.
\end{proposition}

\begin{proof}
Let $\Delta$ be a boundary on $X$ so that
$(X,\Delta)$ is klt and $-(K_X+\Delta)$ is big and nef.
Define $B_\epsilon:= (1-\epsilon)B+\epsilon\Delta$.
Observe that for $\epsilon \in (0,1)$ we have that
$B_\epsilon$ is big, 
$(X,B_\epsilon)$ is klt,
and $-(K_X+B_{\epsilon})$ is big and nef.
Furthermore, for $\epsilon>0$ small enough,
all the prime divisors extracted on $Y$
have log discrepancy in the interval $(0,1)$ with respect to $(X,B)$.
By~\cite{Mor19}*{Theorem 1}, we can find a birational projective morphism
$Z\rightarrow X$ which extract exactly those divisorial valuations.
For now on, we fix such $\epsilon$ small enough.
Let $(Z,B_{Z,\epsilon})$ be the log pull-back 
of $(Y,B_{Y,\epsilon})$ to $Z$.
Note that $(Z,B_{Z,\epsilon})$ is a klt pair
with $-(K_Z+B_{Z,\epsilon})$ big and nef.
Hence, $Z$ is a Fano type variety.
Furthermore, we have a small morphism $Z\dashrightarrow Y$.
We can pick $\Gamma_Z$ a big boundary on $Z$ so that
$(Z,\Gamma_Z)$ is klt and $K_Z+\Gamma_Z\sim_\qq0$.
Let $\Gamma_Y$ be the push-forward of $\Gamma_Z$ to $Y$.
We have that $\Gamma_Y$ is big on $Y$ so that
$(Y,\Gamma_Y)$ is klt and $K_Y+\Gamma_Y\sim_\qq 0$.
We conclude that $Y$ is of Fano type.
\end{proof}

We conclude this subsection by proving two lemmas regarding quotients and covers of Fano type varieties.
We prove that the quotient of a Fano type
variety is again of Fano type.
On the other hand, the cover of a Fano type variety 
is of Fano type provided that the branching locus of the morphism lies in the set of log canonical centers of certain $\qq$-complement.

\begin{lemma}\label{lem:quot-FT-is-FT}
Let $X$ be a Fano type variety and $X\rightarrow Y$ be a Galois cover.
Then $Y$ is a Fano type variety.
\end{lemma}

\begin{proof}
Assume that $\pi\colon X\rightarrow Y$ is the quotient by the finite group $G$ acting on $X$.
Let $\Delta$ be a big boundary on $X$ so that
$(X,\Delta)$ is klt and $K_X+\Delta\sim_\qq 0$.
Then, the big boundary $\Delta^G=\sum_{g\in G}g^*\Delta/|G|$
satisfies that $K_X+\Delta^G \sim_\qq 0$
and $(X,\Delta^G)$ is klt.
Furthermore, there exists a boundary $\Delta_Y$ on $Y$ so that
$K_X+\Delta^G=\pi^*(K_Y+\Delta_Y)$.
Thus, $(Y,\Delta_Y)$ is a klt pair.
By Riemann-Hurwitz formula we have that
$\pi^*(-K_Y)=-K_X+E$,
where $E$ is an effective divisor.
Hence, $\pi^*(-K_Y)$ is a big divisor.
Thus, we can write $\pi^*(-K_Y)\sim_\qq A+F$,
where $A$ is a $G$-invariant ample divisor
and $F$ is a $G$-invariant effective divisor.
Henceforth, the push-forward of $\pi^*(-K_Y)$
with respect to $\pi$ is big as well.
Thus, we have that $-K_Y$ is big, so $\Delta_Y$ is big.
We conclude that $Y$ is a Fano type variety.
\end{proof}

\begin{lemma}\label{lem:cover-FT-is-FT}
Let $Y$ be a Fano type variety and $X\rightarrow Y$ be a Galois cover. 
Assume that there exists a boundary $B_Y$ on $Y$ so that
$K_Y+B_Y\sim_\qq 0$ and $(Y,B_Y)$ is log canonical.
Assume that the ramification divisors
of $X\rightarrow Y$ are contained in the log canonical
centers of $(Y,B_Y)$.
Then, the variety $X$ is of Fano type.
\end{lemma}

\begin{proof}
Let $\Delta_Y$ be a boundary on $Y$ so that
$(Y,\Delta_Y)$ has klt singularities and $-(K_Y+\Delta_Y)$ is big and nef.
Define $\Delta'_{Y,\epsilon}:=(1-\epsilon)B_Y+\epsilon \Delta_Y$.
Note that for every $\epsilon \in (0,1)$,
we have that $K_Y+\Delta_{Y,\epsilon}$ is a klt pair
so that $-K_Y+\Delta_{Y,\epsilon}$ is big and nef.
Furthermore, 
if we pick $\epsilon>0$ small enough,
the pull-back of $K_Y+\Delta_{Y,\epsilon}$
is a log pair
$K_X+\Delta_{X,\epsilon}$.
Indeed, it suffices to take $\epsilon < 1/r$,
where $r$ is the largest ramification index among 
codimension one points of $Y$,
which by assumption are log canonical centers of $(X,B_X)$, so they appear in $\lfloor B_X\rfloor$.
We conclude that, for $\epsilon$ small enough,
the log pair $(X,\Delta_{X,\epsilon})$ has klt singularities
and $-K_X+\Delta_{X,\epsilon}$ is big and nef.
Hence, $X$ is of Fano type.
\end{proof}

\subsection{Toric Geometry}\label{subsec:toric-geometry}
In this subsection, we recall some basics of toric geometry and the characterization of toric varieties using complexity.

\begin{definition}{\em 
A {\em toric variety} is a normal variety $X$
of dimension $n$
containing an algebraic torus $\mathbb{G}_m^n$ as an open subset, 
so that the action of $\mathbb{G}_m^n$ on itself
by multiplication extends to the whole variety.
A divisor $B$ on $X$ is said to be {\em torus invariant} or
{\em toric divisor} if it is invariant with respect
to the action of $\mathbb{G}_m^n$.
A pair $(X,B)$ is said to be a {\em toric pair}
if $X$ is a toric variety and $B$ is a torus invariant divisor.
A morphism $X\rightarrow Y$ between toric varieties is said to be {\em a toric morphism}
if it is equivariant with respect to the corresponding torus actions.
}
\end{definition}

\begin{definition}
{\em 
We say that a point $x\in X$ on an algebraic variety is {\em toroidal} if the completion of its local ring is isomorphic to the completion of the local ring of a point on a toric variety.
For instance, any smooth point is toroidal.
Analogously, we say that a log pair $x\in (X,\Delta)$ is {\em toroidal} if $x\in X$ is toroidal and
the isomorphism identifies $\Delta$ with a toric boundary.
Note that a toroidal pair must have log canonical singularities on a neighborhood of $x$.
}
\end{definition}

The following lemma states that for a projective toric variety $X$ there is a unique toric invariant boundary $B$ so that the pair $(X,B)$ is log Calabi-Yau.

\begin{lemma}\label{lem:unique-tlcy}
Let $X$ be a projective toric variety.
Let $B$ be the complement of the torus on $X$ with reduced scheme structure.
Then $B$ is an effective divisor,
$(X,B)$ has log canonical singularities, 
and $K_X+B\sim 0$.
\end{lemma}

\begin{proof}
The fact that $B$ is a divisor follows from~\cite{CLS11}*{Theorem 3.2.6},
the fact that $(X,B)$ has log canonical divisor is proved in~\cite{Amb06}*{\S 2}, 
and finally the linear equivalence
$K_X+B\sim 0$ is proved in~\cite{CLS11}*{Theorem 8.2.3}.
\end{proof}

The following lemma states that any projective log Calabi-Yau toric pair is crepant equivalent  to the projective space with the union of the torus invariant hyperplanes as the boundary.

\begin{lemma}\label{lem:toric-bir-pn}
Let $(X,B)$ be a $n$-dimensional log Calabi-Yau toric pair.
Then, the birational map
$X\dashrightarrow \pp^n$ is crepant equivalent 
for the log pair structures
$(X,B)$ and $(\pp^n,H_1+\dots+H_{n+1})$.
\end{lemma}

\begin{proof}
Let $\Sigma$ be the fan in $N_\qq$ defining $X$
and $\Sigma_n$ be the standard fan in $N_\qq$ defining $\pp^n$.
Consider a common refinement fan $\Sigma'$ of $\Sigma$ and $\Sigma_n$.
Let $X'$ be the toric variety defined by $\Sigma'$ and $B'$ its toric boundary.
Then, we have induced birational toric morphisms
$\phi\colon X'\rightarrow X$ and
$\phi_n \colon X'\rightarrow \pp^n$, which we may assume to be projective.
Furthermore, we have that
$\phi^*(K_X+B)=K_{X'}+B'$
and 
$\phi_n^*(K_{\pp^n}+H_1+\dots+H_{n+1})=
K_{X'}+B'$.
We conclude that $(X,B)$ and $(\pp^n,H_1+\dots+H_{n+1})$ are crepant equivalent  toric pairs.
\end{proof}

The following proposition will be used to prove that certain birational models of toric projective varieties are still toric.

\begin{proposition}\label{prop:birational-toric}
Let $(X,B)$ be a projective toric pair
and $\phi\colon Y\dashrightarrow X$ be a birational map from a projective variety $Y$.
Assume that all the prime divisors extracted by $\phi$ have log discrepancy in the interval $[0,1)$ with respect to the log pair $(X,B)$.
Denote by $(Y,B_Y)$ the log pull-back of $(X,B)$ to $Y$.
Then, $(Y,B_Y)$ is a toric pair as well.
\end{proposition}

\begin{proof}
It is known that terminal valuations over toric pair are toric valuations as well (see, e.g.~\cite{dFdC16}*{Proposition 6.2}).
Furthermore, any log canonical place of a toric pair is a terminal valuation after perturbing the coefficient of the toric boundary.
Hence, we can find a projective toric morphism
$X'\rightarrow X$ so that
$X'\dashrightarrow Y$ is a birational map which is dominant in codimension one.
Let $A_Y$ be an ample divisor on $Y$
and $A_{X'}$ its strict transform on $X'$.
$A_{X'}$ is linearly equivalent to a toric divisor on $X'$, it is pseudo-effective, and its base locus has codimension at least two.
We may run a minimal model program for $A_{X'}$, which terminates in a good minimal model, since $X'$ is a Mori dream space.
All the steps of this minimal model program are toric and the ample model is toric as well.
Furthermore, the ample model is small birational to $Y$, so we conclude that $Y$ is a projective toric variety as well.
This last statement follows since the complexity doesn't change in small modifications.
\end{proof}

A main application of the above proposition is the following corollary, which allows us to deduce that $X$ is toric when some minimal model program terminates in a toric variety.

\begin{corollary}\label{cor:mmp-ending-toric}
Let $(X,B)$ be a projective log Calabi-Yau pair.
Assume that $X\dashrightarrow X'$ is a minimal model program for some divisor on $X$ that only contract divisors contained in the support of $B$.
Let $B'$ be the push-forward of $B$ to $X'$.
Furthermore, assume that $(X',B')$ is a log Calabi-Yau toric pair.
Then $(X,B)$ is a toric pair as well.
\end{corollary} 

Finally, we state a proposition that will allow us to prove that certain finite automorphism of a projective toric variety is indeed a subgroup of the torus.

\begin{proposition}\label{prop:making-H0-toric}
Let $n$ and $h$ be two positive integers.
There exists a positive integer $N:=N(n,h)$, only depending on $n$ and $h$, satisfying the following.
Let $H_0<{\rm Aut}(\mathbb{G}_m^n)$ be a subgroup of order at most $h$ so that $H_0$ is normalized by $A$, where
$\zz_M^n \leqslant A < \mathbb{G}_m^n < {\rm Aut}(\mathbb{G}_m^n)$
and $M$ is at least $N$.
Then, $H_0$ is a subgroup of $\mathbb{G}_m^n$.
\end{proposition}

\begin{proof}
We may assume that $H_0$ is cyclic generated by $h_0$.
Replacing $A$ by $A^{|{\rm Aut}(H)|}$, we may assume that
$A$ commutes with $h_0$.
Note that $|{\rm Aut}(H)|$ is bounded by $O_h(1)$.
Furthermore, we have that $A^{|{\rm Aut}(H)|} \geq \zz_M^n$ provided that
$A\geq \zz_{M|{\rm Aut}(H)|}^n$.
Therefore, we may assume that $h_0$ commutes with $\zz^n_M$.
Write
\[
h_0(x_1,\dots,x_n)=
(\lambda_1 x_1^{c_{1,1}}\cdots x_n^{c_{1,n}}, \dots ,
\lambda_n 
x_1^{c_{n,1}}\cdots x_n^{c_{n,n}}).
\]
Let $C$ be the matrix in ${\rm GL}_n(\zz)$ with entries $c_{i,j}$'s.
Then, we have that $C^h$
is the identity matrix.
Furthermore, since $h$ commutes with $A$, we conclude that
\[
c_{i,j}\equiv_M 0 \text{ for each $i\neq j$}
\]
and 
\[
c_{i,i}\equiv_M 1 \text{ for each $1\leq i\leq n$.}
\]
Hence, we may write 
$c_{i,j}=Mc'_{i,j}$
and $c_{i,i}=Mc'_{i,i}+1$
for certain integeres $c'_{i,j}$'s.
On the other hand, we know that the eigenvalues of $C$ are $h$-roots of unity.
If $\sum_{i=1}^n c'_{i,i}\neq 0$
and $M>2n$, 
then $||n+M(\sum_{i=1}^n c'_{i,i})||>n$.
In particular, the trace of $C$ can't be equal 
to the sum of at most $n$ roots of unity.
We conclude that
provided $M>2n$,
we have that $\sum_{i=1}^n c'_{i,i}=0$.
From now on, we may assume this is the case.
In particular, we have that the trace of $C$ equals $n$.
However, since the eigenvalues of $C$ are exactly $n$
roots of unity,
this implies that all the eigenvalues must be one.
Since $C$ is diagonalizable and all its eigenvalues are one,
we conclude that $C$ is the identity matrix.

Let $h':=h'(h)$ be an upper bound for the order
of the automorphism group of all finite groups of order at most $h$.
Note this number $h'$ only depends on $h$.
We deduce that
provided $M>\max\{h',2n\}$,
we have that $H_0<\mathbb{G}_m^n$ as claimed.
\end{proof}

The proof of the following proposition is analogous to the one of Proposition~\ref{prop:making-H0-toric}.
We give a proof for the sake of completeness.

\begin{proposition}\label{prop:making-abelian-group-toric}
Let $n$ and $h$ be two positive integers.
Let $H_0<{\rm Aut}(\mathbb{G}_m^n)$ be an abelian group.
Assume $A_0\leqslant H_0$ is a subgroup of index at most $h$ 
satisfying
$\zz_M^n\leqslant  A_0<\mathbb{G}_m^n$
where $M>2n$.
Then, $H_0$ is a subgroup of $\mathbb{G}_m^n$.
\end{proposition}

\begin{proof}
Note that $H_0^h < A_0 < \mathbb{G}_m^n$.
Recall that, we have an exact sequence
\[
1\rightarrow \mathbb{G}_m^n
\rightarrow {\rm Aut}(\mathbb{G}_m^n)
\rightarrow {\rm GL}_n(\zz)
\rightarrow 1.
\]
Let $h_0$ be an element of $H_0$.
The image of $h_0$ in ${\rm GL}_n(\zz)$ is a matrix 
$C$ with $C^h$ being the identity matrix.
Hence, we conclude that the eigenvalues of $C$ are roots of unity.
Write
\[
h_0(x_1,\dots,x_n)=
(\lambda_1 x_1^{c_{1,1}}\cdots x_n^{c_{1,n}}, \dots ,
\lambda_n 
x_1^{c_{n,1}}\cdots x_n^{c_{n,n}}).
\]
Since $h_0$ commutes with $A_0$, then
it also commutes with $\zz_M^n$.
As in the proof of the previous proposition, we conclude that 
\[
c_{i,j}\equiv_M 0 \text{ for each $i\neq j$}
\]
and 
\[
c_{i,i}\equiv_M 1 \text{ for each $1\leq i\leq n$.}
\]
Hence, we may write 
$c_{i,j}=Mc'_{i,j}$
and $c_{i,i}=Mc'_{i,i}+1$
for certain $c'_{i,j}$'s.
On the other hand, we know that the eigenvalues of $C$ are 
roots of unity.
If $\sum_{i=1}^n c'_{i,i}\neq 0$
and $M>2n$, 
then $||n+M(\sum_{i=1}^n c'_{i,i})||>n$.
Leading to a contradiction.
We conclude that
provided $M>2n$,
we have that $\sum_{i=1}^n c'_{i,i}=0$.
From now on, we may assume this is the case.
In particular, we have that the trace of $C$ equals $n$.
However, since the eigenvalues of $C$ are exactly $n$
roots of unity,
this implies that all the eigenvalues must be one.
Since $C$ is diagonalizable and all its eigenvalues are one,
we conclude that $C$ is the identity matrix.
This means that $h_0 \in \mathbb{G}_m^n$, provided that $M>2n$, as claimed.
\end{proof}

\subsection{Complexity}\label{subsec:complexity}
In this subsection, we recall the concept of complexity
and the characterization of projective toric varieties using this invariant.
There are more fine versions of the complexity than the one we use in this article (see, e.g.,~\cite{BMSZ18}).

\begin{definition}{\em  
Let $(X,\Delta)$ be a log pair.
Let $x\in X$ be a closed point.
Let $\Delta=\sum_{i=1}^k d_i\Delta_i$ be the prime decomposition
of $\Delta$ around $x\in X$.
We define the {\em local complexity} of $(X,\Delta)$ at $x$ to be
\[
c(X,\Delta;x):=\dim(X)+\rho_x(X)-\sum_{i=1}^k d_i,
\]
where $\rho_x(X)$ stands for the local Picard rank, i.e., the $\qq$-rank of the local $\qq$-Class group.
}
\end{definition}

The following theorem asserts that the local complexity of a $\qq$-factorial pair is always non-negative.
Furthermore, if the complexity is zero, it is supported on a toroidal point.

\begin{theorem}[cf.~\cite{Kol92}*{Proposition 18.22}]\label{thm:toroidality}
Let $x\in (X,\Delta)$ be a $\qq$-factorial log canonical pair.
Then, $c_x(X,\Delta;x)\geq 0$.
If the equality holds, then $(X,\lfloor D\rfloor)$ is toroidal around $x\in X$.
\end{theorem}

Now, we proceed to introduce a global version of the complexity.

\begin{definition}{\em
Let $(X,\Delta)$ be a projective pair
and $\Delta=\sum_{i=1}^k d_i\Delta_i$ be the prime decomposition of the boundary divisor.
We define the {\em complexity} of $(X,\Delta)$ to be
\[
c(X,\Delta):= \dim(X)+\rho(X)-\sum_{i=1}^k d_i,
\]
where, as usual, $\rho(X)$ is the Picard rank of the projective variety $X$.}
\end{definition}

The following is a version of the main theorem of~\cite{BMSZ18}.
This version will suffice for our purposes in this article.

\begin{theorem}[cf.~\cite{BMSZ18}*{Theorem 1.2}]\label{thm:toric-by-complexity}
Let $(X,\Delta)$ be a projective pair with log canonical singularities so that $K_X+\Delta\sim_\qq 0$.
Then, we have that $c(X,\Delta)\geq 0$.
Furthermore, if $c(X,\Delta)=0$, then
$(X,\lfloor\Delta\rfloor)$ is a toric pair.
\end{theorem} 

Now, we turn to introduce two propositions that explain when the quotient or cover of a toric variety remains toric.
Both of them are well-known facts, we give short proofs using the complexity.

\begin{proposition}\label{prop:quot-toric-by-toric-finite}
Let $X$ be a projective $n$-dimensional toric variety and
$H<\mathbb{G}_m^n\leqslant {\rm Aut}(X)$ 
be a finite group.
Then, the quotient $Y:=X/H$ is a toric variety.
\end{proposition}

\begin{proof}
Let $D_1,\dots,D_k$ be the prime torus invariant divisors of $X$.
Then, we have that $k=\dim(X)+\rho(X)$.
Furthermore, each $D_i$ is $H$-invariant.
Define $B:=\sum_{i=1}^k D_i$ 
and $(Y,B_Y)$ be the log quotient of $(X,B)$.
Then, $(Y,B_Y)$ has log canonical singularities.
Furthermore, we have that
$\rho(Y)\leq \rho(X)$.
Thus, the complexity of $(Y,B_Y)$
is at most $\dim(Y)+\rho(Y)-k \leq \dim(X)+\rho(X)-k=0$.
By Theorem~\ref{thm:toric-by-complexity}, we conclude that $(Y,B_Y)$ is a projective toric pair and $X\rightarrow Y$ is a finite toric morphism.
\end{proof}

\begin{proposition}\label{prop:cover-toric-unramified-torus-toric}
Let $Y$ be a projective toric variety and $X\rightarrow Y$ be a finite morphism which only ramifies over the complement of the torus of $Y$.
Then $X$ is a projective toric variety.
\end{proposition}

\begin{proof}
Set $n=\dim(X)=\dim(Y)$.
Let $B_Y$ the be reduced toric boundary of $Y$ and $D_1,\dots,D_k$ its prime components.
Since $\pi\colon X\rightarrow Y$ is unramified over the torus, we conclude that $X$ contains an open subset isomorphic to the torus.
For each $i$, we let
$D^i_1,\dots,D^i_{j_i}$ be the prime components of $\pi^*(D_i)$.
Henceforth, the Picard group of $X$ is generated by the divisors $D_j^i$ which are prime.
In particular, we have that
$\rho(X)\leq \sum_{i=1}^k j_i$.
Recall that the Picard rank of $Y$ equals $k-n$.
Hence, we can find $k-n$ $\qq$-divisors $\Gamma_1,\dots,\Gamma_{k_n}$, supported on $\supp(D_1\cup\dots\cup D_k)$, which generate
the Picard group.
In particular, for each $1\leq i\leq k$, we can write
\[
D_i\sim_\qq \sum_{j=1}^{k-n} \gamma_{i,j}\Gamma_j.
\]
Pulling-back the above relations via $\pi$, we obtain
$\qq$-linearly equivalences
\[
m_i\left(\sum_{j=1}^{j_i} D^j_s\right) 
=
\pi^*D_i \sim_\qq \sum_{j=1}^{k-n} \gamma_{i,j}\pi^*\Gamma_j.
\]
Hence, 
we have a relation of the form
\[
D^i_1 \sim_\qq 
\sum_{j=1}^{k-n} \frac{\gamma_{i,j}}{m_i}
\pi^*\Gamma_j  -
\sum_{j=2}^{j_i} D^i_j.
\]
Thus, the divisors $\pi^*\Gamma_j$ with $1\leq j\leq k-n$
and the set
$\{ D_j^i \mid j\in \{2,\dots,j_i\}\}$
generate the Picard group of $X$.
Thus, we have that $\rho(X)\leq \sum_{i=1}^k j_i-n$.
Let $(X,B)$ be the log pull-back of $(Y,B_Y)$ to $X$.
Note that $(X,B)$ is log Calabi-Yau 
and $\lfloor B\rfloor$ has 
$\sum_{i=1}^k j_i$ components.
We conclude that 
\[
c(X,B)= \dim(X) + \rho(X) - \sum_{i=1}^k j_i \leq 
n+\left(\sum_{i=1}^k j_i-n\right)-\sum_{i=1}^k j_i = 0.
\]
By Theorem~\ref{thm:toric-by-complexity}, 
we conclude that $(X,B)$ is a toric pair.
\end{proof}

\begin{remark}{\em
In the two above propositions, we obtain that $X\rightarrow Y$ is a finite toric morphism 
and $\rho(X)=\rho(Y)$.
Furthermore, the combinatorial structures of the fans of $X$ and $Y$ are the same.
Although, 
the monomorphism of these fan structures 
into the lattice $N$ may change.}
\end{remark}

\subsection{Log crepant equivalent  toric quotients}
In this subsection, we introduce the concept
of log crepant equivalent  toric quotient projective varieties
and log crepant equivalent  toric quotient singularities.
These are analogs of quotient singularities and toric singularities.

\begin{definition}{\em 
Let $(X,B)$ be a projective pair.
We say that $(X,B)$ is {\em crepant equivalent  toric},
if there exists a birational map $\pi\colon X\dashrightarrow T$ 
to a toric variety $T$ 
so that $\pi^*(K_T+B_T)=K_X+B$,
where $B_T$ is the reduced toric boundary of $T$.
By the pull-back of a birational map, we mean the pull-back to a common resolution and then push-forward to $X$.
Note that in the above case, we will have $K_X+B\sim 0$.
A projective variety $X$ is said to be
{\em log crepant equivalent  toric} if there exists a boundary $B$ on $X$ so that $(X,B)$ is crepant equivalent  toric.
The boundary $B$ will be called the {\em crepant equivalent  toric boundary}.
}
\end{definition}

\begin{definition}{\em 
A projective variety $X$ is called
{\em log crepant equivalent  toric quotient}
if $X$ is the quotient of a log crepant equivalent  toric variety by a finite group $G$.
The push-forward of the crepant equivalent  toric boundary to $X$ is called the {\em crepant equivalent  toric quotient boundary}.}
\end{definition}

\begin{definition}
\label{def:lce-tq}
{\em 
A klt singularity is called {\em toric quotient} if it is the quotient of a toric singularity by a finite group.
Note that all toric singularities and quotient singularities are toric quotient singularities.
A klt singularity is called {\em log crepant equivalent  toric quotient} if it is isomorphic to the orbifold cone over a log crepant equivalent  toric quotient projective variety.
We write {\em lce-tq} singularities for short.
}
\end{definition}

\begin{remark}{\em 
Note that we have the following inclusions of classes of singularities:
\[
\{\text{Cyclic quotient singularities}\}
= 
\{\text{Quotient singularities}\} 
\cap 
\{\text{Toric singularities}\} 
\subsetneq 
\]
\[
\{\text{Quotient singularities}\} 
\cup 
\{\text{Toric singularities}\} 
\subsetneq 
\{\text{Toric quotient singularities}\}
\subsetneq  
\]
\[
\{
\text{Log crepant equivalent  toric quotient singularities}
\}.
\]
The class of lce-tq singularities is the smaller class which arises from cyclic quotient singularities
and is closed with respect to: taking finite quotients, taking universal covers, and taking birational models of the exceptional divisor of a plt blow-up.
Hence, these singularities arises naturally from cyclic quotient singularities by considering 
natural operations on singularities of the MMP.
Other classes of singularities of the minimal model program, for instance, exceptional singularities,
are closed under all these operations as well (see, e.g.,~\cite{Mor18a}). In Example~\ref{example}, we give an example of a lce-tq singularity which is not toric quotient.
}
\end{remark} 

\subsection{Regional fundamental group}\label{subsec:regional-fund}
In this subsection, we recall the concept of regional fundamental group and the fundamental group of the log smooth locus of a Fano type variety.

\begin{definition}{\em 
We say that a log pair $(X,\Delta)$ has {\em standard coefficients}
if the coefficients of $\Delta$ are of the form $1-\frac{1}{n}$,
where $n$ is a positive number (see Definition~\ref{def:hyperstandard}).
Giving a log pair $(X,\Delta)$, we define its {\em standard approximation}
to be the pair $(X,\Delta_s)$ with largest boundary so that
$\Delta_s\leq \Delta$ and $\Delta_s$ has standard coefficients.
}
\end{definition}

\begin{definition}{\em 
Let $x\in (X,\Delta)$ be the germ of a singularity.
We denote by $X^{\rm sm}$ the locus on which $(X,\Delta)$ is log smooth.
Let $U$ be a neighborhood of $x$.
We denote by $\Delta_s$ the standard approximation of $\Delta$.
For each prime component $P$ of $\Delta_s$, 
we denote by $n_P$ the positive integer so that such component
appears with coefficient $1-\frac{1}{n_P}$ in $\Delta_s$.
We define the fundamental group $\pi^{\rm alg}_1(U,\Delta_U;x)$
to be the subgroup of $\pi^{\rm alg}_1(U\cap X^{\rm sm}\setminus \supp(\Delta_U))$
which correspond to finite covers which ramify with index 
at most $n_P$ at each prime $P$.
We define the {\em regional fundamental group} of $x\in (X,\Delta)$,
denoted by $\pi_1^{\rm reg}(X,\Delta;x)$, to be the inverse limit 
of all the groups $\pi^{\rm alg}_1(U,\Delta_U;x)$ where the $U$'s run over all open neighborhoods of $x$.
In the case that we work with complex germs, we may replace $\pi^{\rm alg}$ with the usual fundamental group $\pi^{\rm loc}$.}
\end{definition} 

It is known that the regional fundamental group of klt singularities is finite.
In the case of the algebraic fundamental group, this is due to Chenyang Xu~\cite{Xu14}.
In the case of complex singularities and the usual fundamental group, this is a recent result of Lukas Braun~\cite{Bra20}. In the next subsection, we turn to give a more precise characterization of such groups. 
The following proposition is the existence of universal covers for the regional fundamental group.

\begin{proposition}\label{prop:existence-universal-cover}
Let $x\in (X,\Delta)$ be the germ of a klt singularity.
Then, there exists a Galois cover $\pi\colon Y\rightarrow X$,
which is the quotient by $G:=\pi_1^{\rm reg}(X,\Delta;x)$,
so that $x\in X$ has a unique pre-image
and the pull-back $\pi^*(K_X+\Delta)=K_Y+\Delta_Y$ is a klt germ.
\end{proposition}

\subsection{Jordan Property}\label{subsec:jp}
In this subsection, we recall the Jordan property for the automorphism group of Fano type varieties, and for the regional fundamental group of klt singularities.
The theorems of this section are obtained relying on the work of Prokhorov and Shramov~\cites{PS14,PS16}.

\begin{definition}
{\em 
Let $\mathcal{G}$ be a class of finite groups.
We say that $\mathcal{G}$ satisfy the Jordan property with rank $k$
if there exists a constant $J:=J(\mathcal{G})$, only depending on $\mathcal{G}$, so that
for every group $G\in \mathcal{G}$, we can find an abelian normal subgroup
$A\leqslant G$ of index at most $J$ and rank at most $k$.
}
\end{definition}

The following theorems are due to Prokhorov and Shramov.

\begin{theorem}[cf.~\cite{PS16}*{Theorem 4.2}]\label{thm:fixed-point-RC}
Let $n$ be a positive integer.
There exists a function $N(n)$, only depending on $n$, satisfying the following.
Let $G$ be a finite group acting on a rationally connected variety of dimension $n$.
Then, $G$ admits a subgroup of index at most $N(n)$
which acts with a fixed point.
\end{theorem}

\begin{theorem}\label{thm:Jordan-FT}
The class $\mathcal{G}$ of finite automorphisms of Fano type varieties of dimension $n$
satisfy the Jordan property with rank $n$.
\end{theorem}

\begin{definition}{\em 
We say that a pair $(X,\Delta)$ is a {\em Fano type pair} if there exists
$B\geq \Delta$ so that $(X,B)$ has klt singularities and $-(K_X+B)$ is big and nef.
The {\em log smooth locus} of a Fano type pair is the largest open subset on which $(X,\Delta)$ is log smooth.
}
\end{definition} 

The following theorem is proved in~\cite{BFMS20}.
It is one of the main pieces towards proving the Jordan property
for the regional fundamental group of klt singularities.

\begin{theorem}[cf.~\cite{BFMS20}*{Theorem 3}]\label{thm:Jordan-log-smooth-proj}
The fundamental group of the log smooth locus of Fano type pairs is finite.
Furthermore, the class of fundamental groups of the log smooth locus of Fano type pairs of dimension $n$
satisfy the Jordan property with rank $n$.
\end{theorem}

The following theorem is the Jordan property for the regional fundamental group
of klt singularities. It is stated in terms of complex germs.
However, a similar proof works over algebraically closed fields of characteristic zero.

\begin{theorem}[cf.~\cite{BFMS20}*{Theorem 2}]\label{thm:jordan-klt}
The class of regional fundamental group of complex klt singularities of dimension $n$
satisfy the Jordan property with rank $n$.
\end{theorem}

It is well-known that the regional fundamental group surjects into the local fundamental group of the singularity, so we obtain the following theorem.

\begin{theorem}[cf.~\cite{BFMS20}]
The class of local fundamental groups of complex klt singularities
of dimension $n$ satisfy the Jordan property with rank $n$.
\end{theorem}

The above theorem was conjectured by Shokurov.
It is also conjectured that a better bound of the rank can be given in terms of regularity of the klt singularity (see, e.g.,~\cite{Sho00}*{\S 7}).
As mentioned above, this article aims to understand how the existence
of a large abelian group in the regional fundamental group of a klt singularity reflects in its geometry. The class of lce-tq singularities is a prototype of singularities
with large abelian subgroups of rank $n$ in their regional fundamental group.

\subsection{Groups with large $k$-generation}\label{subsec:k-gen}
In this subsection, we recall the definition of the $k$-generation of a finite group and study some properties of finite groups with large $k$-generation.
This concept is already introduced in~\cite{Mor20}.
Recall that we take the rank of a group to be the minimum cardinality of a set of generators.

\begin{definition}
{\em 
We define the {\em rank of up to index $N$}
of a group $G$ to be minimum rank among subgroups of $G$ of index at most $N$.
We denote the rank up to index $N$ by $r_N(G)$.
If $N=0$, then $\rank(G)$ is the usual rank of a finite group, i.e.,
the minimum number of generators.
}
\end{definition}

\begin{definition}
{\em 
Let $G$ be a finite group.
Let $k$ be a positive integer which is bounded above by the rank of $G$.
We define the {\em $k$-generation order} of $G$ to be
\[
g_k(G):=
\max\left\{
N \mid |G|\geq N
\text{ and }
r_N(G)\geq k
\right\}.
\]
Note that $g_k(G)\geq 1$ for each $k\leq \rank(G)$.
Furthermore, having large $1$-generation is the same as having large cardinality.
}
\end{definition}

The following proposition implies that large $k$-generation
is a hereditary property for subgroups of bounded index.

\begin{proposition}\label{prop:k-gen-bounded-subgroup}
Let $G$ be a finite group with $g_k(G)\geq N$
Let $H$ be a subgroup of index at most $J$ with $J<N$.
Then $H$ has $g_k(H)\geq N/J$.
\end{proposition}

\begin{proof}
Note that $\rank(H)\geq k$
and $|H|\geq N/J$.
Assume that $g_k(H)<N/J$.
Then, we have that $r_{N/J}(H)<k$.
So, there exists a subgroup $A$ of $H$
of index at most $N/J$ which is generated by less than $k$ elements.
Hence, $A$ is a subgroup of $G$ of index at most $N$
which is generated by less than $k$ elements.
This contradicts the fact that $r_N(G)\geq k$.
Thus, we conclude that $g_k(H)\geq N/J$.
\end{proof}

The following lemma shows that abelian groups of rank at most $n$
with large $n$-generation contains $\zz_m^n$ with some large $m$.

\begin{lemma}\label{lem:bounded-rank-large-k-generation}
Let $A$ be an abelian group of rank at most $n$ 
with $g_n(A)\geq N$.
Then, $A$ contains $\zz_N^m$ as a subgroup
for some $m\geq N$.
\end{lemma}

\begin{proof}
Write $A\simeq \zz_{m_1}\oplus\dots\oplus\zz_{m_n}$
where $m_1\mid m_2\mid \dots \mid m_n$.
Since $g_n(A)\geq N$, we have that 
$m_1\geq N$.
So, we conclude that $\zz_{m_1}^n\leqslant A$ with $m_1\geq N$.
\end{proof}

We will use the following lemma to study the induced actions
on the fiber and base of an equivariant Mori fiber space.

\begin{lemma}\label{lem:ses-z_m^n}
Consider an exact sequence
\[
1\rightarrow A_F \rightarrow A \rightarrow A_C \rightarrow 1.
\]
Assume that $g_n(A)\geq N$.
Then, the following conditions hold:
\begin{enumerate} 
\item If $A_F$ has a subgroup of index at most $M$ and rank at most $f$, then $g_{n-f}(A_C)\geq N/M$.
\item If $A_C$ has a subgroup of index at most $M$ and rank at most $c$, then $g_{n-c}(A_F)\geq N/M$.
\end{enumerate}
\end{lemma}

\begin{proof}
We prove the first statement.
Assume this is not the case, then $A_C$ has a subgroup $H_C$ of index at most $N/M$ 
which is generated by less than $n-f$ elements.
We can pick elements $a_1,\dots,a_{n-f-1}\in A$ which generate a subgroup surjecting onto $H_C$.
Let $H_F$ be the subgroup of $A_F$ of index at most $M$ which is generated by at most $f$ elements.
We conclude that $\langle a_1,\dots,a_{n-f-1},H_F\rangle$ is a subgroup of $A$ of index at most $N$ which is generated by less than $n$ elements. This leads to a contradiction.

Now, we prove the second statement.
Let $H_C$ be the subgroup of $A_C$ of index at most $M$ and rank at most $c$.
Let $a_1,\dots,a_c$ be elements of $A$ which maps to the generators of $H_C$.
Assume that $g_{n-c}(A_F)<N/M$.
Then, there exists a subgroup $H_F$ of $A_F$ of index at most $N/M$ which is generated by elements
$h_1,\dots, h_{n-c-1}$.
Consider the group $H=\langle 
a_1,\dots,a_c,h_1,\dots,h_{n-c-1}\rangle$.
Then, $H$ is a subgroup with less than $n$ generators so that its index on $A$ is at most $N$.
This gives a contradiction of the fact that $g_n(A)\geq N$.
\end{proof}

The following lemma will be used in the proof of our main local statement
to prove that the induced action on a plt center has large $k$-generation.

\begin{lemma}\label{lem:quot-cycl-k-gen}
Let $G$ be a finite group with $g_n(G)\geq N$
and $n\geq 2$.
Consider an exact sequence
\[
1\rightarrow C \rightarrow G\rightarrow H\rightarrow 1,
\]
where $C$ is a cyclic group.
Then, we have that $g_{n-1}(H) \geq N$.
\end{lemma}

\begin{proof}
Assume that $g_{n-1}(H)<N$.
If $|H|<N$, then $C$ is a cyclic subgroup of index at most $N$ on $G$,
which leads to a contradiction of $g_n(G)\geq N$.
On the other hand, assume that
$g_{n-1}(H)<N$.
Then, we can find a subgroup $H'\leqslant H$ of index at most $N$
so that $H'$ is generated by less than $n-1$ generators.
Let $h'_1,\dots, h'_s$ be the generators of $H'$ with $s\leq n-1$.
Let $g_1,\dots,g_s$ be elements of $g$ mapping to $h'_1,\dots, h'_s$ respectively.
Let $g$ be the image of the generator of $C$ on $G$.
Then, we have that $g,g_1,\dots,g_s$ are less than $n$ elements
generating a subgroup of $G$ of index at most $N$.
This leads to a contradiction as well.
We conclude that $g_{n-1}(H)\geq N$.
\end{proof}

\subsection{$G$-equivariant MMP}\label{subsec:g-equiv-mmp}
In this subsection, we recall some basic results about the $G$-equivariant minimal model program.

\begin{definition}{\em 
Let $X$ be an algebraic variety and $G\leqslant {\rm Aut}(X)$ be a finite group.
We say that $X$ is {\em $G\qq$-factorial} if every $G$-invariant divisor on $X$ is $\qq$-Cartier.
}
\end{definition}

The two following propositions are straightforward applications of~\cite{BCHM10}*{Corollary 1.4.3}.
We give a proof of the first statement for the sake of completeness.
The proof of the second statement is analogous.

\begin{proposition}\label{prop:existence-g-term}
Let $(X,\Delta)$ be a $G$-invariant klt pair.
Then, there exists a projective birational $G$-equivariant morphism
$\pi\colon X'\rightarrow X$ which extract the divisorial valuations
with log discrepancy in $(0,1]$.
Furthermore, $X'$ is $G\qq$-factorial.
\end{proposition}

\begin{proof}
Let $(Y,\Delta_Y)$ be the log quotient of the pair $(X,\Delta)$ with respect to $G$.
Note that $(Y,\Delta_Y)$ is a klt pair.
Furthermore, all the divisorial valuations of $(X,\Delta)$ with log discrepancies
in $(0,1]$ with respect to $(X,\Delta)$, will correspond to 
divisorial valuations of $(Y,\Delta_Y)$ with log discrepancies in $(0,1]$ with respect to $(Y,\Delta_Y)$. Let $\mathcal{V}$ be such finite set of valuations over $(Y,\Delta_Y)$.
By~\cite{BCHM10}*{Corollary 1.4.3}, we know that there exists a projective birational morphism $Y'\rightarrow Y$ which is $\qq$-factorial and
extract divisors corresponding to $\mathcal{S}$.
Now, let $X'$ be the normalization of the main component of $Y'\times_Y X$.
Then, we have that $X'$ is $G\qq$-factorial and it extracts all divisorial valuations
with log discrepancy in $(0,1]$ with respect to $(X,\Delta)$.
By construction, the birational morphism $X'\rightarrow X$ is $G$-equivariant.
\end{proof}

The model constructed in the above proposition will be called a {\em $G$-terminalization} of the $G$-invariant pair.

\begin{proposition}\label{prop:existence-g-can}
Let $(X,\Delta)$ be a $G$-invariant klt pair.
Then, there exists a projective birational $G$-equivariant morphism 
$\pi\colon Y\rightarrow X$ which extract the divisorial valuations
with log discrepancy $(0,1)$.
Furthermore, $Y$ is $G\qq$-factorial.
\end{proposition} 

The model constructed in the above proposition will be called a {\em 
$G$-canonicalization} of the $G$-invariant pair. 

\begin{definition}\label{def:hyperstandard}
{\em 
Let $\mathcal{S}$ be a finite set of rational numbers.
We define the {\em hyperstandard set of coefficients} associated to $\mathcal{S}$, denoted by $\Phi(\mathcal{S})$, to be the set
$\left\{ 1-\frac{s}{m} \mid s\in \mathcal{S}, m\in \mathbb{N}\right\}$.
If $\mathcal{S}:=\{1\}$, then we say that $\Phi(\mathcal{S})$ is the set of standard coefficients.
Note that $\Phi(\mathcal{S})$ satisfies the descending chain condition, i.e., 
there is no infinite decreasing sequence on $\Phi(\mathcal{S})$.
}
\end{definition}

\begin{proposition}\label{prop:g-equiv-cbf}
Let $n$ and $M$ be two positive integers.
There exists a constant $q:=q(n,M)$, only depending on $n$ and $M$,
satisfying the following.
Let $X$ be a Fano type variety of dimension $n$.
Let $G\leqslant {\rm Aut}(X)$ be a finite subgroup.
Let $(X,B)$ be a log canonical pair which is $G$-invariant
and $M(K_X+B)\sim 0$.
Let $\pi\colon X\rightarrow C$ be a $G$-equivariant contraction.
Then, we can write
\[
q(K_X+B) \sim q\pi^*(K_C+B_C+M_C),
\]
where the coefficients of $B_C$ belong to $\nn\left[q^{-1}\right]$
and $M_C$ is the push-forward of a nef divisor $M_{C'}$ on a projective birational model $C'\rightarrow C$ so that $qM_{C'}$ is Cartier.
Furthermore, $B_C$ and $M_C$ are equivariant with respect to the action induced by $G$ on the base.
\end{proposition}

\begin{proof}
We have an exact sequence
\[
1\rightarrow G_f\rightarrow G \rightarrow G_C\rightarrow 1, 
\]
where $G_f$ is acting fiber-wise
and $G_C$ is the induced action on the base.
We consider the quotient of the fibration
$X\rightarrow C$ by $G$ obtaining a commutative diagram as follows:
\[
\xymatrix{
X  \ar[r]^-{f}\ar[d]_-{\pi} & Y\ar[d]^-{\pi_Y} \\
C \ar[r]^-{f_C} & C_Y.
}
\]
Let $(Y,B_Y)$ be the quotient of $(X,B)$ by $G$.
Then, we have that $f^*(K_Y+B_Y)=K_X+B$
and the coefficients of $B_Y$ belong to $\Phi\left(\nn\left[M^{-1}\right]\right)$.
By Lemma~\ref{lem:quot-FT-is-FT}, we know that both $Y$ and $C_Y$ are Fano type varieties.
Hence, the contraction $\pi_Y$ satisfies all the conditions of~\cite{Bir19}*{Proposition 6.3}.
So, we can write
\[
q(K_Y+B_Y)\sim q\pi_Y^*(K_{C_Y}+B_{C_Y}+M_{C_Y}),
\]
where $q$ only depends on $M$ and $n$.
The coefficients of $B_{C_Y}$ belong to $\Phi(\mathcal{S})$,
where $\mathcal{S}$ is a finite set of rational numbers which only depend on $M$ and $n$.
Furthermore, $M_{C_Y}$ is the push-forward of a nef $\qq$-Cartier divisor $M_{C_Y'}$ on a projective birational model $C'_Y\rightarrow C_Y$ so that $q M_{C_Y'}$.
Note that $K_{C_Y}+B_{C_Y}+M_{C_Y}\sim_\qq 0$.
By~\cite{BZ16}*{Theorem 1.6}, we conclude that the coefficients of $B_{C_Y}$ are indeed in a finite set.
Furthermore, by the boundedness of complements~\cite{Bir19}*{Theorem 1.7}, we may assume that $q(K_{C_Y}+B_{C_Y}+M_{C_Y})\sim 0$.
We define $K_C+B_C+M_C=f_C^*(K_{C_Y}+B_{C_Y}+M_{C_Y})$.
Then, we conclude that the coefficients of $B_C$ belong to a finite set which only depends on $M$ and $n$.
The same holds for the Cartier index of the nef $\qq$-Cartier divisor  $M_{C'}$.
By construction, $(C,B_C+M_C)$ is a $G$-equivariant pair and we have
\[
q(K_X+B)\sim q\pi^*(K_C+B_C+M_C).
\]
\end{proof}

\begin{proposition}\label{prop:ramification-rowndown}
Let $(X,B)$ be a log canonical pair and $\pi\colon X\rightarrow Y$ be a contraction so that $K_X+B\sim_{\qq,Y} 0$.
Let $B_Y$ be the boundary part of the pair induced by the canonical bundle formula.
Let $S$ be a prime component of $\lfloor B \rfloor$ which ramifies over the prime divisor $P\subset Y$.
Then, we have that $P\subset \supp(B_Y)$.
\end{proposition}

\begin{proof}
Let $(Y,B_Y+M_Y)$ be the generalized pair induced on the base by the canonical bundle formula (see, e.g.,~\cite{Amb06}).
Let $S^\nu\rightarrow S$ be the normalization of $S$.
Consider the Stein factorization of $S^\nu \rightarrow Y$.
We denote by $S^\nu \rightarrow Y'$ the contraction morphism
and $f\colon Y'\rightarrow Y$ the finite morphism.
Let $K_{S^\nu}+B_{S^\nu}$ be the induced pair on $S^\nu$.
By construction $(S^\nu,B_{S^\nu})$ has log canonical singularities.
We apply the canonical bundle formula to $K_{S^\nu}+B_{S^\nu}$ with respect to $Y'$
and we obtain a generalized pair $(Y',B_{Y'}+M_{Y'})$.
By construction, we have that
$f_*(K_{Y'}+B_{Y'}+M_{Y'})= K_{Y}+B_{Y}+M_{Y}$ is the push-forward with respect to a finite morphism.
Hence, if $f$ ramifies along a codimension one point, then $B_Y$ has a positive coefficient at such prime divisor.
\end{proof}

\subsection{$G$-equivariant complements}\label{subsec:g-equiv-comp}
In this subsection, we recall some results about the existence of bounded $G$-complements.

\begin{definition}{\em 
Let $(X,\Delta)$ be a log pair.
We say that a boundary $B\geq \Delta$ on $X$ is a {\em $\qq$-complement}
of $(X,\Delta)$
if $(X,B)$ has log canonical singularities and $K_X+B\sim_\qq 0$.
We say that a boundary $B\geq \Delta$ on $X$ is a {\em $M$-complement}
of $(X,\Delta)$
if $(X,B)$ has log canonical singularities and $M(K_X+B)\sim 0$.
In the former case, we say that $-(K_X+\Delta)$ is {\em $\qq$-complemented}.
In the latter case, we say that $-(K_X+\Delta)$ is
{\em $M$-complemented}.
}
\end{definition}

\begin{definition}{\em 
Let $(X,\Delta)$ be a log pair
and $G\leqslant {\rm Aut}(X,\Delta)$ be a finite subgroup.
We say that a boundary $B\geq \Delta$ on $X$ is a {\em $G$-equivariant $\qq$-complement} of $(X,\Delta)$ 
if $(X,B)$ is $G$-invariant, has log canonical singularities, and
$K_X+B\sim_\qq 0$.
We say that a boundary $B\geq\Delta$ on $X$ is a {\em $G$-equivariant $M$-complement}
of $(X,\Delta)$
if $(X,B)$ is $G$-invariant, has log canonical singularities, and
$M(K_X+B)\sim 0$.
In the former case, we say that $-(K_X+\Delta)$ is {\em $G\qq$-complemented}.
In the latter case, we say that $-(K_X+\Delta)$ is 
{\em $G$-equivariantly $M$-complemented}.
}
\end{definition}

It is well-known that a $G$-invariant $\qq$-complemented pair is also
$G\qq$-complemented (see,e.g.,~\cite{Mor20}*{Proposition 2.17}).

\begin{theorem}[cf.~\cite{Mor20}*{Theorem 2.19}]\label{thm:G-equiv-complement}
Let $n$ be a positive integer and $\Lambda\subset \qq$ with $\overline{\Lambda}\subset \qq$ which satisfies the descending chain condition.
There exists a constant $M:=M(n,\Lambda)$, only depending on $n$ and $\Lambda$, 
satisfying the following.
Let $X$ be a $n$-dimensional Fano type variety and $\Delta$ be a boundary on $X$ 
with the following conditions:
\begin{enumerate}
\item $G\leqslant {\rm Aut}(X,\Delta)$ is a finite group, 
\item $(X,\Delta)$ is $G$-invariant and log canonical,
\item the coefficients of $\Delta$ belong to $\Lambda$, and
\item $-(K_X+\Delta)$ is $\qq$-complemented.
\end{enumerate}
Then, $(X,\Delta)$ admits a $G$-equivariant $M$-complement.
\end{theorem}

\section{Bounded anti-pluricanonically embededd varieties}\label{sec:bounded}
In this section, we study bounded families of anti-pluricanonically embedded varieties which have large finite automorphisms.
In this case, we prove that a bounded index subgroup of such finite groups factors through a torus action. The main idea is to use the boundedness of Fano type varieties
to reduce the problem to a finite set of isomorphisms classes of connected reductive groups.

\begin{theorem}\label{thm:bounded-anti-pluricanonically-embedded-varieties}
Let $n$ and $M$ be positive integers.
There exists a constant $M':=M'(M,n)$, only depending on $M$ and $n$, satisfying the following.
Let $(X,B)$ be a $n$-dimensional log canonical pair so that
$M(K_X+B)\sim 0$ and
$X$ is a canonical Fano variety.
If $A$ is an abelian group with $g_n(A)\geq M'$ so that $A<{\rm Aut}(X,B)$, then
$X$ is a projective toric variety
and $B$ is the reduced toric boundary.
Furthermore,
we have that $A<\mathbb{G}_m^n\leqslant {\rm Aut}(X,B)$.
\end{theorem}

\begin{proof}
Since $X$ is canonical Fano, then the algebraic groups ${\rm Aut}(X)$ form a bounded family of linear algebraic groups.
Given that $M(K_X+B)\sim 0$, we conclude that
the log pairs $(X,B)$ are log bounded.
Hence, the subgroups
${\rm Aut}(X,B)\leqslant {\rm Aut}(X)$ form a bounded family of linear algebraic groups.
The number of connected components of ${\rm Aut}(X,B)$ is upper semicontinuous over the base.
Hence, by Noetherian induction, we deduce that such a number must achieve a finite upper bound on this bounded family.
In particular, there exists a constant $M'_0:=M'_0(M,n)$, only depending on $M$ and $n$, so that
$M'_0 \geq |{\rm Aut}(X,B)/{\rm Aut}^0(X,B)|$ for every log pair $(X,B)$ as in the statement of the theorem.
Let $A_0$ be the intersection of $A$ with ${\rm Aut}^0(X,B)$.
By Proposition~\ref{prop:k-gen-bounded-subgroup}, we know that $g_n(A_0)\geq g_n(A)/M_0'$.
Let ${\rm Aut}_R(X,B)$ be a connected reductive linear algebraic quotient of ${\rm Aut}^0(X,B)$.
Note that the groups ${\rm Aut}_R(X,B)$ have bounded dimension.
Recall that connected reductive linear algebraic groups of bounded dimension belong to
finitely many isomorphisms classes.
Then, we conclude that the groups
${\rm Aut}_R(X,B)$ belong to finitely many isomorphisms classes of reductive groups.
Let $\mathcal{R}$ be such finite set of isomorphism classes
of connected reductive linear algebraic groups.
Furthermore, any split maximal algebraic torus of ${\rm Aut}_R(X,B)$ lifts to a split algebraic torus of ${\rm Aut}^0(X,B)$ (see, e.g.,~\cite{Bor91}*{Theorem 10.6.(4)}).
Since the unipotent radical does not contain finite subgroups, we conclude that there is a monomorphism
$A_0 < {\rm Aut}_R(X,B)$.
It suffices to prove that
for $g_n(A_0)\geq M'$, where $M':=M'(\mathcal{R})$ only depends on $\mathcal{R}$, the finite group $A_0$ is contained in a maximal split torus of ${\rm Aut}_R(X,B)$.
Since $\mathcal{R}$ is finite, we may take $A_0$ depending on each such reductive group $G\in R$.

Let $G$ be a connected reductive linear algebraic group.
We claim that there exists $M':=M'(G)$, only depending on $G$, satisfying the following: If $A_0<G$ is a finite abelian group with
rank at most $n$ and $g_n(A_0)\geq M'$, then
a bounded subgroup of
$A_0$ is contained on a split torus of $G$ of rank at least $n$.
Let $B$ be a Borel subgroup of $G$.
Then $A_0$ acts on $G/B$ which is a projective rationally connected variety.
By Theorem~\ref{thm:fixed-point-RC}, we know that a bounded subgroup $A_0'$ of $A_0$ acts with a fixed point.
The index of $A_0'$ on $A_0$ is bounded by a function ofly depending on $\dim(G/B)$ which only depends on $G$.
Hence, $A_0'$ is contained in a Borel subgroup of $G$.
Then, since $A_0'$ only contains semi-simple elements, it must be contained in a maximal torus of such Borel subgroup.
By Lemma~\ref{lem:bounded-rank-large-k-generation},
we conclude that
$A_0'\geqslant \zz_M^n$, whenever $g_n(A_0')\geq M$.
In particular, such maximal torus has rank at least $n$.
Thus, we obtain an induced torus of rank at least $n$
in ${\rm Aut}(X,B)$.
Since $X$ is a $n$-dimensional algebraic variety, we conclude that it must be a projective toric variety.
By construction, $B$ must be invariant under the torus action and effective.
Hence, we conclude that $B$ must be the reduced toric boundary.

Thus, we have a subgroup $A_0'\leqslant A$ of index bounded by a function $M(n)$ on $n$
so that $A_0'<\mathbb{G}_m^n$.
We claim that $A<{\rm Aut}(\mathbb{G}_m^n)$.
Indeed, by assumption, we know that 
$A<{\rm Aut}(X,B)$ where $(X,B)$ is a projective toric pair.
In particular, $A$ must act as an automorphism group
of $X\setminus \supp \lfloor B\rfloor \simeq \mathbb{G}_m^n$.
This proves the claim.
Note that $A_0'<\mathbb{G}_m^n$ has rank at most $n$
and $g_n(A_0')>g_n(A)/M(n)$.
By Lemma~\ref{lem:bounded-rank-large-k-generation}, 
we conclude that $A_0'\geqslant \zz^n_M$ with $M\geq N$
provided that $g_n(A)\geq NM(n)$.
Then, all the conditions of Proposition~\ref{prop:making-abelian-group-toric}
are satisfied.
We conclude that for
$g(A)_n\geq 2nM(n)$, 
we have that $A<\mathbb{G}_m^n$.
This proves the last statement of the theorem.
\end{proof}

\begin{remark}{\em 
The above theorem holds in general for any family of log pairs $(X,B)$ so that
$-K_X$ is ample,
$X$ has $\epsilon$-log canonical singularities for some $\epsilon>0$,
and the coefficients of $B$ are bounded away from zero.
However, we state it for canonical Fano varieties since we only use it in this context.
Note that in the theorem we start with a bounded family of Fano varieties, and we end up with a finite family of Fano toric varieties~\cite{BB92}.}
\end{remark}

\section{Degenerate divisors}\label{sec:deg}
In this section, we recall the concept of degenerate divisors from~\cite{Lai11}.
We state classic results which prove that components of a degenerate divisors are covered by curves which intersect it negatively.

\begin{definition}\label{def:degenerate}
{\em 
Let $X\rightarrow Y$ be a proper surjective morphism of normal varieties.
An effective $\qq$-divisor $D$ on $X$ is called {\em degenerate over $Y$} if the following conditions hold:
\begin{enumerate}
    \item No component of $D$ is horizontal over the base, and 
    \item for every prime divisor $P_Y\subset Y$ there exists a prime divisor $P_X\subset X$ so that $P_X\subsetneq \supp(D)$ and
$f(P_X)=P_Y$.
\end{enumerate}
Note that, if the image of $D$ on $Y$ has codimension at least two, 
then it is automatically degenerate.
Hence, the degenerate condition must be checked over codimension one points of the image of $D$.
}
\end{definition}

Note that our definition of degenerate divisor is slightly different from the definition in~\cite{Lai11}*{Definition 2.8}. In our case, we try to put both concepts of {\em $f$-exceptional} and {\em insufficient fiber type} together.
However, by the following proposition, the main property of degenerate divisors still holds.

\begin{proposition}\label{prop:divisor-trivial-over-basis}
Let $f\colon X\rightarrow Y$ be a projective morphism so that $Y$ is normal and $X$ has klt singularities.
Let $D$ be an effective divisor on $X$ which is degenerate over $Y$.
Then, there is a prime component $F\subset \supp D$ which is covered by curves which are contracted by $f$ and intersect $D$ negatively.
In particular, we have that 
$F\subset {\rm Bs}_{-}(D/Y)$.
\end{proposition}

\begin{proof}
Note that we can always replace $X$ by a small $\qq$-factorialization and $D$ with a multiple.
In particular, we may assume $D$ is a Weil divisor.
Assume that the image of $D$ on $Y$ contains a prime divisor $P_Y$.
Then, by assumption, we know that there is a prime divisor $P_X$ mapping to $P_Y$ which is not contained in the support of $D$.
We claim that we can find a component $F$ as in the statement of the proposition which dominates $P_Y$.
Cutting by hyperplanes on the base, we may assume that we have a morphism from a surface to a curve.
Since the statement is local over $Y$, we may always shrink an assume that $D$ is contained in the fiber.
We may write $D=\sum r_i C_i$ where the $r_i$ are positive integers and the $C_i$ are possibly non-reduced and reducible curves that don't contain the whole fiber.
If $D^2=0$, then $D$ must be supported on the whole fiber, leading to a contradiction.
Otherwise, $D^2<0$ and we have a component 
$C$ of some curve $E_i$ for which $D\cdot C<0$.
Then, the corresponding component of $D$ is covered by curves contracted by $f$ which intersect $D$ negatively.

In the case that the image of $D$ on $Y$ has codimension at least two, then this is proved in~\cite{Lai11}*{Lemma 2.9}.
\end{proof}

The following proposition proves that a relative minimal model program for a degenerate divisor terminates with a good minimal model on which all components of the degenerate divisor are contracted.

\begin{proposition}\label{prop:MMP-for-deg}
Let $f\colon X\rightarrow Y$ be a projective morphism so that $Y$ is normal and $X$ has klt singularities.
Let $D$ be an effective divisor on $X$ which is degenerate over $Y$.
Then, any minimal model program for $D$ with scaling of an ample divisor over the base terminates after contracting all components of $D$. 
\end{proposition}

\begin{proof}
Note that $D$ is an effective divisor so
${\rm Bs}_{-}(D)\subset D$ and
the minimal model program must terminate, if it does, with a minimal model.
By Proposition~\ref{prop:divisor-trivial-over-basis}, we know that after finitely many steps, some component $F\subset \supp D$ will be contracted.
Note that after such contraction, the strict transform of $D$ remains degenerate over the base in the sense of Definition~\ref{def:degenerate}.
Therefore, we can proceed inductively and deduce that eventually all components of $D$ will be contracted, so the minimal model program will terminate with a $\qq$-trivial good minimal model over $Y$.
\end{proof}

The following lemma is a consequence of
the main property of degenerate divisors.

\begin{lemma}\label{lem:vert-div-Q-trivial}
Let $f\colon X\rightarrow Y$ be a projective morphism so that $Y$ is normal and $X$ has klt singularities.
Let $D$ be an effective prime divisor on $X$ which is vertical over $Y$
and is $\qq$-linearly trivial over $Y$.
Then, the image of $D$ on $Y$ is a prime divisor.
\end{lemma}

\begin{proof}
If the image of $D$ on $Y$ has codimension at least two, then by Proposition~\ref{prop:divisor-trivial-over-basis}, we would have that $D$ has non-trivial diminished base locus over the base.
Leading to a contradiction.
We conclude that the image of $D$ on $Y$ is a prime divisor.
\end{proof}

The following lemma is a refinement of Lemma~\ref{lem:vert-div-Q-trivial}.
In particular, it says that in a Mori fiber space
every vertical prime divisor maps to a prime divisor.

\begin{lemma}\label{lem:vert-div-MFS}
Let $f\colon X\rightarrow Y$ be a projective morphism so that $Y$ is normal and $X$ has klt singularities.
Assume that the relative Picard rank of $f$ equals one.
Let $D$ be a prime divisor on $X$ which is vertical over $Y$.
Then, the divisor $D$ is the pull-back of a prime divisor on $Y$.
\end{lemma}

\begin{proof}
Note that $D$ is $\qq$-linearly trivial over $Y$ since it intersects every curve on a general fiber of $f$ trivially.
By Lemma~\ref{lem:vert-div-Q-trivial}, we concldue that the image of $D$ on $Y$ is a prime divisor $D_Y$.
Since $D$ is $\qq$-trivial over the base, we conclude that the support of $D$ equals the set-theoretic pre-image of $D_Y$.
Hence, the pull-back of some positive multiple of $D_Y$ equals $D$.
\end{proof}

The following lemma allows us to argue that a prime vertical divisor is not contracted on a minimal model program.
It is a straightforward application of the upper-semicontinuity of fiber dimensions.

\begin{lemma}\label{lem:not-contracted-prime-vertical}
Let $X\rightarrow Y$ be a projective morphism from a klt variety $X$ to a normal variety $Y$.
Let $\pi\colon X\rightarrow X'$ be a contraction over $Y$.
Let $F$ be a prime effective divisor on $X$ which is the pull-back of a divisor on $Y$.
Then $F$ is not contracted by $\pi$.
\end{lemma}

The following proposition will help us to prove that certain vertical prime divisors on a Galois cover of a Mori fiber space are the pull-back of a prime divisor on the base.

\begin{proposition}\label{prop:galois-cover-uncontractibility}
Let $X\rightarrow Y$ be a projective morphism of relative Picard rank one.
Assume that $X$ has klt singularities and $Y$ is normal.
Let $Y'\rightarrow Y$ be a Galois cover which is \'etale over $U_Y\subset Y$.
Let $U_{Y'}$ be the pre-image of $U_Y$ on $Y'$.
Let $X'$ be the normalization of the main component of $Y'\times_Y X$.
Assume that $Y'$ is $\qq$-factorial.
Then, every vertical prime divisor on $X'$ whose image on $Y'$ intersect $U_{Y'}$ 
is the pull-back of a prime divisor on $Y'$.
\end{proposition}

\begin{proof}
Let $Y'\rightarrow Y$ be the quotient by $G$.
Hence, $X'\rightarrow X$ is the quotient by $G$ as well.
The morphism $X'\rightarrow X$ is \'etale over the pre-image of $U_Y$.
Let $F$ be a vertical prime divisor which maps to $U_{Y'}$.
We claim that $F$ is the pull-back of a prime divisor on $Y'$.
Let $F^G:=\sum_{g\in G}g^*F$.
Then $F^G$ is the pull-back of a prime divisor $F_X$ on $X$.
This prime divisor is vertical over $Y$, so
it is $\qq$-linearly trivial over $Y$,
given that $X\rightarrow Y$ has relative Picard rank one.
By Lemma~\ref{lem:vert-div-MFS}, we conclude that $F_X$ is the pull-back of a prime divisor $F_Y$ on $Y$.
Denote by $F_{Y'}$ the pull-back of $F_Y$ to $Y'$.
We conclude that $F^G$ is the pull-back to $X'$ of $F_{Y'}$.
By construction, we conclude that each prime component of $F^G$ is the pull-back of a prime component of $F_{Y'}$.
This proves the claim.
\end{proof}

\section{Proof of the main theorems}
\subsection{Proof of the main projective theorem}\label{subsec:proof-main-projective}
In this subsection, we prove the main projective theorem of this article.
First, we start by stating a more general version of the main projective theorem.
At the end of this subsection, we show a version of the following theorem which considers possibly non-abelian actions.

\begin{theorem}\label{thm:FT-full-rank}
Let $n$ be a positive integer and
let $\Lambda\subset \qq$ be a set satisfying the descending chain condition with rational accumulation points.
Then, there exists a positive integer
$M:=M(\Lambda,n)$ 
so that $M$ only depends on $\Lambda$ and $n$,
satisfying the following.
Let $X$ be a Fano type variety of dimension $n$ and $\Delta$ be a boundary on $X$, such that the following conditions hold:
\begin{enumerate}
\item $A\leqslant {\rm Aut}(X)$ is a finite abelian subgroup with $g_n(A)\geq M$, 
\item $(X,\Delta)$ is log canonical and $A$-invariant,
\item the coefficients of $\Delta$ belong to $\Lambda$, and
\item $-(K_X+\Delta)$ is $\qq$-complemented.
\end{enumerate}
Then, there exists:
\begin{enumerate}
\item A boundary $B\geqslant \Delta$ on $X$, and
\item an $A$-equivariant birational map $X\dashrightarrow X'$, 
\end{enumerate}
satisfying the following conditions:
\begin{enumerate}
\item The pair $(X,B)$ is log canonical, $A$-equivariant, and 
$(K_X+B)\sim 0$,
\item the push-forward of $K_X+B$ to $X'$ is a log pair $(X',B')$,
\item the pair $(X',B')$ is a log Calabi-Yau toric pair, and
\item there are group monomorphisms
$A<\mathbb{G}_m^n\leqslant {\rm Aut}(X,B)$.
\end{enumerate}
In particular, $B'$ is the reduced toric boundary of $X'$.
Furthermore, the birational map $X\dashrightarrow X'$ is an isomorphism over $\mathbb{G}_m^n$.
\end{theorem}

Before proceeding to the proof,
note that Theorem~\ref{introthm:main-thm},
Corollary~\ref{introcor-easy-version}, 
and Corollary~\ref{introcor1} follow directly from the above theorem.
Indeed, we can set $\Lambda=\{0\}$
and consider the given $M$,
which in this case will only depend on the dimension.

\begin{proof}
We prove the theorem in several steps.
By induction, we may assume the statement holds in dimension at most $n-1$.\\

\textit{Step 0:} In this step, we prove the case of dimension one.\\

In this case $X\simeq \pp^1$.
If $g_1(A)=|A|>4$, then $A$ acts as the multiplication by a roof of unity on $\pp^1$.
Up to an automorphism, we may assume it fixes zero and infinite.
Let $\delta>0$ be so that every positive element
of $\Lambda$ is larger than $\delta$.
If $g_1(A)=|A|>\frac{4}{\delta}$, then
no component of $\Delta$ can be contained in $\mathbb{G}_m\subset \pp^1$.
Otherwise, the divisor $K_X+\Delta$ is ample.
We conclude that $\Delta$ is supported on zero
and infinite.
Hence, we can take $B$ to be the sum of zero and infinite with reduced scheme structure.
Then, the statement of the theorem holds for 
$M=\max\{4,4/\delta\}+1$.
This proves the statement in dimension one.\\

\textit{Step 1:} In this step, we construct a bounded $A$-invariant complement for $(X,\Delta)$ and
run an $A$-equivariant minimal model program.\\

Note that the log pair $(X,\Delta)$ satisfies all the hypotheses of Theorem~\ref{thm:G-equiv-complement}.
Hence, there exists an $A$-equivariant log canonical $M_0$-complement $B\geq \Delta\geq 0$.
The pair $(X,B)$ is log canonical, $A$-invariant,
and $M_0(K_X+B) \sim 0$. Here, $M_0$ only depends on $n$ and $\Lambda$.
By Proposition~\ref{prop:existence-g-can}, we can produce an $A$-equivariant canonicalization $X'\rightarrow X$ of $X$.
Note that the log pull-back $(X',B')$ of $(X,B)$ to $X'$ satisfies that $(X',B')$ is log canonical, $A$-invariant, and $M_0(K_{X'}+B')\sim 0$.
By Proposition~\ref{prop:FT-extraction-non-canonical}, 
$X'$ is a Fano type variety. 
Now, we run an $A$-equivariant minimal model program for $K_{X'}$.
We conclude that 
this $A$-equivariant minimal model program terminates with an $A$-equivariant Mori fiber space
$X''\rightarrow C$.
Note that $X''$ is a Fano type variety with canonical singularities.
We denote by $B''$ the push-forward of $B'$ to $X''$.
Observe that we have $(X'',B'')$ is log canonical, 
$A$-invariant, and $M_0(K_{X''}+B'')\sim 0$, 
where $M_0$ only depends on $n$ and $\Lambda$.\\

\textit{Step 2:} In this step, we prove the first statement of the 
theorem in the case that $C$ is a point.\\

If $C$ is a point, then the $A$-invariant Picard rank of $X''$ is one.
Since $K_{X''}$ is an $A$-invariant divisor and is not pseudo-effective, we conclude that
$K_{X''}$ is anti-ample.
Hence, $X''$ is canonical with anti-ample canonical divisor.
Thus, $X''$ is a canonical Fano variety.
We conclude that $(X'',B'')$ belongs to a log bouded family (see Theorem~\ref{thm:log-ft-bdness}).
Hence, by Theorem~\ref{thm:bounded-anti-pluricanonically-embedded-varieties}, we conclude that $X''$ is a projective toric variety and $B''$ is the toric boundary whenever $g_n(A)\geq M_1(M_0,n)$.
Furthermore, provided that $g_n(A)\geq M_1(M_0,n)$, we have that $A<\mathbb{G}_m^n\leqslant {\rm Aut}(X'',B'')$.
Note that $M_1$ only depends on $M_0$ and $n$, which at the same time only depend on $M$ and $n$. 
Hence, $M_1$ only depends on $M$ and $n$.
We reduce to prove that the 
birational map $X''\dashrightarrow X$ is an isomorphism over the torus. This is the content of the next step.\\

\textit{Step 3:} In this step, assuming that $C$ is a point, we prove that the birational map
$X''\dashrightarrow X$ is an isomorphism over $\mathbb{G}^n_m$.\\

Observe that $K_X+B$ is the log pull-back of $K_{X''}+B''$ to $X$.
Hence, the divisor $K_X+B$ is linearly equivalent to zero.
In particular, every log discrepancy of $K_{X}+B$ is 
a positive integer.
Note that, we can factor $X\dashrightarrow X''$ as
$X \leftarrow X' \dashrightarrow X''$, where
$X'\rightarrow X$ is a canonicalization of $X$ and
$X'\dashrightarrow X''$ is a minimal model program for $K_{X'}$.
We can factorize the minimal model program in a sequence
of birational maps
\[
(X',B')=(X'_i,B'_i) \dashrightarrow
(X'_{i-1},B'_{i-1})\dashrightarrow
(X'_{i-2},B'_{i-2})\dashrightarrow \dots
\dashrightarrow 
(X'_1,B'_1)=(X'',B'').
\]
Each step
$(X'_{j+1},B'_{j+1})\dashrightarrow (X'_j,B'_j)$ is either
a $K_{X'_{j+1}}$-flip or a $K_{X'_{j+1}}$-divisorial contraction.
Note that each step
$(X'_{j+1},B'_{j+1})\dashrightarrow (X'_j,B'_j)$
is either a
$(K_{X'_{j+1}}+B'_{j+1})$-flop
or a 
$(K_{X'_{j+1}}+B'_{j+1})$-crepant equivalent  divisorial contraction.
By induction on $j$, we prove that 
$(X_j',B'_j)\dashrightarrow (X'',B'')$ is an isomorphism over
the torus $\mathbb{G}_m^n$
and no component of $B'_j$ intersects the torus.
The statement is clear for $j=1$.
Assume that $(X'_{j+1},B'_{j+1})\rightarrow (X'_j,B'_j)$ is a divisorial contraction. 
Since every log discrepancy with respect $(X'_j,B'_j)$ whose center
intersect the torus $\mathbb{G}_m^n$
must be at least two,
we conclude that 
the center of any divisor contracted by
$X'_{j+1}\rightarrow X'_j$ must be disjoint from the torus.
In particular, 
it is an isomorphism over $\mathbb{G}_m^n$ and
any component of $B'_{j+1}$ remains disjoint
from the torus.
On the other hand,
assume that $(X'_{j+1},B'_{j+1}) \dashrightarrow (X'_j,B'_j)$ is a flop.
Note that such flop is a $K_{X'_{j+1}}$-flip, 
hence it is a $-B'_{j+1}$-flip.
In particular, any flipped curve must be contained in the support
of $B'_j$.
Hence, the flipped locus must be contained in the support of $B'_j$.
Thus, $X'_{j+1}\dashrightarrow X'_j$ is an isomorphism over the torus. Hence, any component of $B'_{j+1}$ remains disjoint 
from the torus.
We conclude that $X''\dashrightarrow X'$ is an isomorphism over the torus
and all the components of $B'$ are disjoint from the torus.
Finally, observe that the birational contraction
$X'\rightarrow X$ only extract divisors with log discrepancy
in the interval $(0,1)$ with respect to $K_X$.
Since $K_X+B\sim 0$ and
$(X',B')\dashrightarrow (X,B)$ is log crepant equivalent, 
we conclude that $X'\rightarrow X$ only extract divisors 
with log discrepancy equal to zero with respect to $(X,B)$.
Thus, $(X',B')\rightarrow (X,B)$ only contract divisorial log canonical places of $(X,B)$, i.e., components of $\lfloor B'\rfloor$.
Since $\lfloor B'\rfloor$ is disjoint from the torus on $X'$,
we conclude that $X'\rightarrow X$ is an isomorphism over $\mathbb{G}_m^n$.
Hence, we have that 
$X''\dashrightarrow X$ is an isomorphism over $\mathbb{G}_m^n$ as claimed.
This concludes the proof in the case that the minimal model program terminates with a trivial Mori fiber space.

From now on, we may assume that $X''\rightarrow C$ is a non-trivial $A$-equivariant Mori fiber space, i.e., 
the base $C$ is a positive dimensional Fano type variety.\\

\textit{Step 4:} In this step, we study the induced actions on the base and the general fibers of $X''\rightarrow C$ and apply the inductive hypothesis.\\

Consider the exact sequence
$1\rightarrow A_F \rightarrow A \rightarrow A_C \rightarrow 1$,
where $A_F$ is the normal subgroup of $A$ acting fiber-wise
and $A_C$ is the induced action on the base.
Let $f$ be the dimension of a general fiber and $c$ the dimension of the base.
Provided that $g_n(A)\geq N$, we have that
$g_f(A_F)\geq N/M_c$ and $g_c(A_C)\geq N/M_f$.
This follows from
Theorem~\ref{thm:Jordan-FT}
and 
Lemma~\ref{lem:ses-z_m^n}.
Here, $M_c$ and $M_f$ are constants that only depend on $c$ and $f$, respectively.
Hence, they only depend on $n$.
Thus, we have that $g_f(A_F)$ and $g_c(A_C)$ are larger than $N$ provided
that $g_n(A)$ is larger than $NM_fM_c$.
In particular, we can apply the inductive hypothesis on the general fiber.

Let $(F,B_F)$ be the log pair induced by $(X'',B'')$ on 
a general fiber $F$ of the $A$-equivariant Mori fiber space
$X''\rightarrow C$.
Observe that $F$ is a canonical Fano variety,
$M_0(K_F+B_F)\sim 0$, and $(F,B_F)$ is log canonical
Moreover, we have that $A_F \leqslant {\rm Aut}(F,B_F)$.
$A_F$ satisfies $g_f(A_F)\geq N$, provided that
$g_n(A)\geq NM_fM_c$.
By Theorem~\ref{thm:bounded-anti-pluricanonically-embedded-varieties}, we know that there exists a constant $N_f:=N_f(M_0,f)$, which only depends on $M_0$ and $f$,
so that $(F,B_F)$ is a log Calabi-Yau toric pair 
provided $g_f(A_F)\geq N_f$.
Note that $N_f$ only depends on $n$ and $\Lambda$.
Furthermore, in such case $A_F$ acts on $F$
as the multiplication by roots of unity.
We conclude that
if $g_n(A)\geq N_fM_fM_c$, then 
$(F,B_F)$ is log toric and $A_F<\mathbb{G}_m^f$.
In particular, the pair $(X'',B'')$ has a log canonical
center which admits a finite morphism to $C$.
By~\cite{FG14}*{Remark 4.1}, the moduli part of the $A$-equivariant canonical bundle formula is $\qq$-trivial.

Note that $C$ is a Fano type variety since is the image of a Fano type variety $X''$ with respect to a contraction.
Let $(C,B_C)$ be the log pair obtained by the $A$-equivariant canonical bundle formula (see Proposition~\ref{prop:g-equiv-cbf}).
Then, we have that $q(K_C+B_C)\sim 0$ for a constant
$q$ which only depends on $n$ and $M_0$.
Hence, it only depends on $n$ and $\Lambda$.
Furthermore, $(C,B_C)$ is log canonical
and we have $A_C\leqslant {\rm Aut}(C,B_C)$.
By the inductive hypothesis, 
if $g_c(A_C)\geq N_c(q,c)$, 
then $(C,B_C)$ admits a log crepant equivalent torus
action which is compatible with $A_C$.
Here, $N_c$ only depends on $q$ and $c$.
Thus, $N_c$ only depends on $n$ and $\Lambda$.
If $g_n(A)\geq N_cM_fM_c$, then
$(C,B_C)$ admits such log crepant equivalent torus action
compatible with $A_C$.

We conclude that
if $g_n(A)\geq N_fN_cM_fM_c$, 
then the two following conditions hold:
\begin{enumerate}
\item $(F,B_F)$ is a log Calabi-Yau toric pair
and $A_F<\mathbb{G}_m^f$, and
\item $(C,B_C)$ admits a log crepant torus action compatible with $A_C$.
\end{enumerate}
From now on, we assume that
$g_n(A)\geq N_fN_cM_fM_c$.
In particular, we assume that the two above conditions are satisfied.
By construction, this value only depends on $n$
and $\Lambda$.\\

\textit{Step 5:} In this step, we construct a $A\qq$-factorial dlt model of
$(X'',B'')$ which admits a fibration to $\pp^c$.\\

Since $(C,B_C)$ admits a log crepant equivalent  torus action,
we know that there exists an $A_C$-equivariant birational map
$\phi\colon \pp^c \dashrightarrow C$.
Moreover, this birational map $\phi$ is an isomorphism over $\mathbb{G}_m^c$.
We have an $A$-equivariant diagram as follows: 
\[
\xymatrix{
X''\ar[d]_-{\pi}  & X''_1\ar[l]\ar[d] \\
C & \pp^c\ar@{-->}[l] \\
}
\]
where $X''_1$ is an $A$-equivariant snc resolution of $(X'',B'')$ (see~\cite{AW97}*{Theorem 0.1}).
Let $B''_1$ be the strict transform of $B''$ on $X''_1$
plus the reduced exceptional divisor of $X''_1\rightarrow X''$.
We denote by $B_{\mathbb{P}^c}$ the toric boundary of the projective space.
We may assume that over each prime component of $B_{\mathbb{P}^c}$ there is a log canonical place of $(X'',B'')$ which is a divisor on $X_1''$.
By construction, $K_{X_1''}+B''_1$ is $\qq$-linearly equivalent to an effective divisor supported on the union of all the divisors 
extracted on $X_1''$ with positive log discrepancy with respect to $(X'',B'')$.
Indeed, we can write
\[
K_{X_1''}+B_1''\sim_\qq 
\pi^*(K_{X''}+B'')+
\sum_{E\ \text{ prime over $X''$}}
a_E(X'',B'')E. 
\]
We define
\[
\Gamma:=\sum_{E\ \text{ prime over $X''$}}
a_E(X'',B'')E.
\]
We write $\Gamma=\Gamma_{\rm vert}+\Gamma_{\rm hor}$
for the decomposition of $\Gamma$ into its vertical
and horizontal components over $\pp^c$.
Note that
\begin{equation}\label{eq:supset}
{\rm Bs}_{-}(K_{X_1''}+B_1''/\pp^c)
\subset \Gamma_{\rm vert}+\Gamma_{\rm hor}.
\end{equation} 
We claim that $\Gamma_{\rm vert}$ is degenerate over $\pp^c$.
Indeed, if $P\subset \lfloor B_{\pp^c}\rfloor$ is prime,
then there is a component of $\lfloor B_1''\rfloor$ mapping
onto $P$ which is not contained in the support of $\Gamma_{\rm vert}$.
On the other hand, let $P\not\subset \lfloor B_{\pp^c}\rfloor$.
Let $P_C$ be the strict transform of $P$ on $C$.
Then, there is a component of $\pi^*P_C$ which maps onto $P_C$.
The strict transform of such component on $X_1''$ is a prime 
divisor on $X_1''$, not contained in the support of $\Gamma_{\rm vert}$, which maps onto $P$.
Furthermore, due to equation~\ref{eq:supset},
for any MMP of $K_{X''_1}+B_1''$ over $\pp^c$
the strict transform of $\Gamma_{\rm vert}$ will remain
degenerate over the base.
Indeed, no vertical component of $X_1''\rightarrow \pp^c$
which is not contained in $\supp(\Gamma_{\rm vert})$
can be contracted in this MMP.

Applying the negativity lemma to the general fiber, we have that
\begin{equation}\label{eq:subset} 
\Gamma_{\rm hor}\subset {\rm Bs}_{-}(K_{X_1''}+B_1''/\pp^c).
\end{equation} 
We run an $A$-equivariant minimal model program for
$K_{X''_1}+B''_1$ over the base $\pp^c$ with scaling of an ample $A$-invariant divisor.
By the inclusion~\ref{eq:supset}, we know that every divisor contracted by this minimal model program is contained in the support of $\Gamma$.
By inclusion~\ref{eq:subset}, we know that, after finitely many steps, all horizontal components are contracted.
So, after finitely many steps,
we are running a minimal model program for the strict transform of $\Gamma_{\rm vert}$.
Note that this is a degenerate divisor.
Then, after finitely many steps, 
this minimal model program is also an MMP
with scaling for a degenerate divisor.
By Proposition~\ref{prop:MMP-for-deg}, we conclude that this minimal model program terminates.
Indeed, it will terminate after contracting all the components of $\Gamma$.
We call such model $(X_2'',B''_2)$.
Hence, we have that $K_{X_2''}+B''_2$ is $\qq$-linearly trivial over the base $\pp^c$.
Note that $X_2''\dashrightarrow X''$ only extract log canonical places
of $(X'',B'')$. 
In particular, $X_2''$ is a Fano type variety as well.
By construction, $(X_2'',B_2'')$ is an $A\qq$-factorial dlt model
of $(X'',B'')$ which admits a fibration to $\pp^c$.\\

\textit{Step 6:} In this step, we will reduce to the case of an equivariant Mori fiber space over the projective space $\pp^c$.\\

By construction, over each hyperplane $H_i$ of $B_{\pp^c}$,
there is an $A$-prime divisorial log canonical center of $(X_2'',B_2'')$
which dominates $H_i$.
Let $H_1,\dots,H_{c+1}$ be the prime components of $B_{\pp^c}$.
Denote by $B_{H_1},\dots,B_{H_{c+1}}$ 
the corresponding $A$-prime divisorial log canonical center.
For each $i\in \{1,\dots,c+1\}$,
we run an $A$-equivariant $-B_{H_i}$-MMP over $\pp^c$ which terminates with a good minimal model over the base.
Note that any such MMP is an isomorphism over the complement of the corresponding $H_i$.
Furthermore, after each MMP terminates, 
the strict transform of $B_{H_i}$
becomes the $A$-equivariant pull-back of $H_i$.
Denote by $X_3''$ the model obtained after all such minimal model programs terminate.
Let $B_3''$ be the induced boundary.
Note that $X_3''$ is still of Fano type
and $X_2''\dashrightarrow X_3''$ is an isomorphism over the pre-image of the torus $\mathbb{G}_m^c$.
The complement of the pre-image of the torus on $X_3''$
is contained on the vertical components of $\lfloor B_3''\rfloor$.
Every vertical $A$-invariant component of $\lfloor B_3''\rfloor$
is the pull-back of some $H_i$ on $\pp^c$.
We have a diagram as follows:
\[
\xymatrix{
X''\ar[d]_-{\pi}  & X''_1\ar[l]\ar[d] \ar@{-->}[r] & X_2''\ar[ld] \ar@{-->}[r] & X_3''\ar[lld]^-{\pi_3} \\
C & \pp^c\ar@{-->}[l] \\
}
\]
Note that over the pre-image of $\mathbb{G}_m^c$ the birational map
$X_3''\dashrightarrow X''$ is a birational contraction.
Furthermore,
$\pi^{-1}(\mathbb{G}_m^c)$
and $\pi_3^{-1}(\mathbb{G}_m^c)$
are Mori dream spaces over $\mathbb{G}_m^c$.
Let $\Omega''$ be an $A$-invariant ample divisor on $X''$.
Let $\Omega''_3$ be its pull-back to $X_3''$.
We run an $A$-equivariant $\Omega''_3$-MMP on $X_3''$ over $\pp^c$
which terminates with a good minimal model.
Then, we take its ample model over $\pp^c$ and call it $X_4''$.
Note that $X_4''$ and $X''$ are isomorphic over the pre-image of $\mathbb{G}_m^c$.
Indeed, over the pre-image of the torus, both varieties are ample models of $\Omega''$ relative to $\mathbb{G}_m^c$.
We denote by $B_4''$ the strict transform of $B_3''$ on $X_4''$.
We let $X^*$ be a small $A\qq$-factorialization of $X_4''$.
Denote by $(X^*,B^*)$ the log pull-back of $(X_4'',B_4'')$ to $X^*$.
Note that $X^*$ is a Fano type variety.
Furthermore, since $X_4''$ is $A\qq$-factorial over the pre-image of $\mathbb{G}_m^c$,
we conclude that $X^*\dashrightarrow X''$ is an isomorphism over the pre-image of the torus $\mathbb{G}_m^c$.
We claim that the fibration $X^*\rightarrow \pp^c$ is an $A$-equivariant Mori fiber space.
Indeed, the pre-image of the torus $\mathbb{G}_m^c$ on $X^*$ has $A$-equivariant Picard rank one. Since such open subvariety is isomorphic 
to the pre-image of the torus $\mathbb{G}_m^c$ on $X''$.
On the other hand, every vertical $A$-prime divisor of $\lfloor B^*\rfloor$
is the pull-back of a divisor on $\pp^c$.
We conclude that $X^*\rightarrow \pp^c$ has
relative $A$-equivariant Picard rank one.
Then, $X^*\rightarrow \pp^c$ is an $A$-equivariant Mori fiber space.
We have an $A$-equivariant commutative diagram as follows:
\[
\xymatrix{
X & X'\ar[l]\ar@{-->}[r] &
X''\ar[d]_-{\pi}  & X^*\ar@{-->}[l]\ar[d] \\
& & C & \pp^c\ar@{-->}[l] \\
}
\]
so that the following conditions are satisfied:
\begin{enumerate}
\item $(X'',B'')$ and $(X^*,B^*)$ are log crepant equivalent  pairs, 
\item $X''$ and $X^*$ are Fano type varieties,
\item $X''\rightarrow C$ and $X^*\rightarrow \pp^c$ are
$A$-equivariant Mori fiber spaces, 
\item $\pp^c\dashrightarrow C$ is an isomorphism over the torus, 
\item $X^*\dashrightarrow X''$ is an isomorphism over the pre-image of the torus, and
\item the pull-back of each prime divisor of $\lfloor B_{\pp^c}\rfloor$ to $X^*$ is supported in a unique $A$-prime divisorial log canonical center of $(X^*,B^*)$,
\end{enumerate}

\textit{Step 7:} In this step, we introduce a projective variety $Y$ which is obtained from $X^*$ quotienting by $A$.\\

Let $Y\rightarrow T_Y$ be the Mori fiber space obtained by quotienting 
the $A$-equivariant Mori fiber space $X^*\rightarrow \pp^c$ by $A$.
We have a commutative diagram as follows:
\[
\xymatrix{
X^*\ar[r]^-{/A}\ar[d] & Y\ar[d] \\
\pp^c \ar[r]^-{/A_C} & T_Y.
}
\]
By Proposition~\ref{prop:quot-toric-by-toric-finite}, $T_Y$ is a projective toric variety of Picard rank one
and $\pp^c\rightarrow T_Y$ is a finite toric morphism.
We denote by $B_{T_Y}$ the toric boundary of $T_Y$.
By~\cite{Amb05}*{Theorem 3.3}, we know that $Y\rightarrow T_Y$ is isotrivial over an open set of the base.
We denote by $(F_Y,B_{F_Y})$ the log general fiber.
Note that $(F_Y,B_{F_Y})$ is obtained by a toric quotient 
from $(F,B_F)$.
In particular, by Proposition~\ref{prop:quot-toric-by-toric-finite}, we conclude that $(F_Y,B_{F_Y})$ is log toric. 
Furthermore, the number of components of $B_{F_Y}$ is $O_f(1)$.
By Proposition~\ref{prop:ramification-rowndown}, prime components of $\lfloor B_Y\rfloor$ can only ramify over the pre-image of $\lfloor B_{T_Y}\rfloor$.
By Lemma~\ref{lem:vert-div-MFS}, we conclude that any vertical prime  divisor on $Y$ is the pull-back of a prime divisor on $T_Y$.
By Lemma~\ref{lem:quot-FT-is-FT}, we know that $Y$ is a Fano type variety. In particular, it is a Mori dream space.\\

\textit{Step 8:}
In this step, we introduce a projective variety $Z$ obtained from $Y$ by taking Galois covers.\\

Let $S$ be a prime component of $\lfloor B_Y\rfloor$ which is horizontal over $T_Y$.
Denote by $S^\nu$ the normalization of $S$. Consider the Stein factorization
of $S^\nu \rightarrow T_Y$.
We write $S^\nu \rightarrow S_Y\rightarrow T_Y$ for such factorization.
By Proposition~\ref{prop:ramification-rowndown}, $S_Y\rightarrow T_Y$ is a finite morphism which is unramified over the torus
of $T_Y$. By Proposition~\ref{prop:cover-toric-unramified-torus-toric},
we conclude that $S_Y$ is toric and $S_Y\rightarrow T_Y$ is a toric Galois cover.
We replace $Y$ by the normalization of the main component of $Y\times_{T_Y}S_Y$.
We proceed inductively, until all such Stein factorizations have trivial finite part.
We call the resulting model $Z$.
By Lemma~\ref{lem:cover-FT-is-FT}, we know that $Z$ is a Fano type variety.
Recall that $F_Y$ is a toric variety
and the number of prime components of $B_{F_Y}$ is $O_f(1)$.
Hence, the degree of the Galois morphism $Z\rightarrow Y$ is $O_f(1)$ as well.
We let $(Z,B_Z)$ be the log pull-back of $(Y,B_Y)$ to $Z$.
We obtain a diagram as follows:
\[
\xymatrix{
X & X'\ar[l]\ar@{-->}[r] &
X''\ar[d]  & X^*\ar@{-->}[l]\ar[d]\ar[rd]_-{/A} & & Z\ar[d]\ar[ld]^-{/H} \\
& & C & \pp^c\ar@{-->}[l]\ar[rd]_-{/A_C} & Y\ar[d] & T_Z\ar[ld]^-{/H} \\
& & & &  T_Y
}
\]
where $H$ is a finite group whose order is $O_f(1)$.\\

\textit{Step 9:} 
In this step, we replace $(Z,B_Z)$ by a pair $(Z',B_{Z'})$ over $T_Z$
whose general fiber is $\pp^f$.\\

We may replace $Z$ with a small $\qq$-factorialization.
This does not change the Picard rank of the general fiber.
The log general fiber $(F_Z,B_{F_Z})$ is isomorphic to the log fiber
$(F_Y,B_{F_Y})$.
Let $c(F)$ be the number of components of $\lfloor B_{F_Z}\rfloor$.
By construction, $c(F)$ is $O_f(1)$.
Furthermore, $\lfloor B_Z\rfloor$ has $c(F)$ horizontal components
that restrict bijectively to those of $\lfloor B_{F_Z}\rfloor$.
Moreover, the pre-image of $B_{T_Z}$ is contained in
the set of log canonical centers of $(Z,B_Z)$. 
By Proposition~\ref{prop:galois-cover-uncontractibility}, 
every vertical prime divisor of $Z\rightarrow T_Z$
which intersect $\mathbb{G}_m^c$ 
is the pull-back of a prime divisor of the base.
In particular, no vertical divisor maps to a codimension $\geq 2$ subvariety of the torus. Furthermore, the fibers over codimension one points of $\mathbb{G}_m^c$ are irreducible.
By Lemma~\ref{lem:toric-bir-pn}, we have two projective toric morphisms
inducing log crepant equivalent  structures:
\[
(F_Z,B_{F_Z})\leftarrow (F_{Z_0},B_{F_{Z_0}})\rightarrow 
(\pp^f,B_{\pp^f}).
\]
The variety $Z$ is toroidal
at general points of intersection
of horizontal components of $\lfloor B_Z\rfloor$ 
(see, e.g., Theorem~\ref{thm:toroidality}).
We can perform a sequence of horizontal toroidal blow-ups on $Z$.
After such sequence of extractions, 
we may assume the log general fiber is isomorphic to $(F_{Z_0},B_{F_{Z_0}})$.
We call this model $Z_0$.
Denote by $(Z_0,B_{Z_0})$ the log pull-back of $(Z,B_Z)$ to $Z_0$.
Every vertical prime divisor of $Z_0\rightarrow T_Z$
which intersects $\mathbb{G}_m^c$ 
is the pull-back of a divisor on $T_Z$.
Indeed, this hold since such vertical divisors on $Z_0$
are pull-back of vertical divisors on $Z$.
Let $H_f$ be a torus invariant hyperplane in $\pp^f$.
Let $H_{F_{Z_0}}$ be the pull-back of $H_f$ to $F_{Z_0}$.
Observe that $H_{F_{Z_0}}$ is a semiample divisor on $F_{Z_0}$
whose ample model is $\pp^f$.
We consider a divisor $H_{Z_0}$ whose restriction to the general fiber equals $H_{F_{Z_0}}$.
Note that $Z_0$ is a Mori dream space, so every minimal model program on $Z_0$ terminates.
We run a minimal model program for $H_{Z_0}$ over the base.
By Lemma~\ref{lem:not-contracted-prime-vertical}, we know that this minimal model program does not contract vertical divisor over $\mathbb{G}_m^c$.
On the other hand, if this minimal model program contract a horizontal divisor,
then such divisor will intersect the general fiber.
This contradicts the fact that $H_{Z_0}$ is semiample on a general fiber.
We conclude that the only divisors that this minimal model program can contract
are vertical components of $\lfloor B_{Z_0}\rfloor$.
Thus, this minimal model program can only contract log canonical places of $(Z_0,B_{Z_0})$.
It must terminate with a good minimal model and it is an isomorphism over an open set of the base.
We denote by $Z_1$ the ample model over $T_Z$.
Observe that
$Z_1$ is a Fano type variety.
We denote by $Z'$ a small $\qq$-factorialization. 
We denote by $B_{Z'}$ the induced boundary on $Z'$.
Then, the log general fiber of $(Z',B_{Z'})$ is isomorphic to
$(\pp^f,B_{\pp^f})$,
where $B_{\pp^f}$ is the toric boundary of $\pp^f$.
Every vertical prime divisor of $Z'$ 
which maps to $\mathbb{G}_m^c\subset T_Z$
is the pull-back of a prime divisor of the base,
after possibly restricting to $\mathbb{G}_m^c$.
Indeed, every such vertical divisor
is the push-forward of a vertical divisor on $Z_0$
which is the pull-back of a divisor on $T_Z$.
Note that the log crepant equivalent  birational map
$(Z,B_Z)\dashrightarrow (Z',B_{Z'})$ only extract log canonical places of $(Z',B_{Z'})$.
By Proposition~\ref{prop:birational-toric}, we conclude that $(Z,B_Z)$ is a toric pair provided that $(Z',B_{Z'})$ is a toric pair.
This is the content of the next step.\\

\textit{Step 10:} We prove that $(Z',B_{Z'})$ is a toric pair.\\

Observe that $T_Z$ has Picard rank one.
We run a minimal model program for $K_{Z'}$ over $T_Z$.
This minimal model program terminates with a Mori fiber space over $T_Z$.
By Lemma~\ref{lem:not-contracted-prime-vertical}, the only vertical divisors contracted by this Mori fiber space are those contained on the vertical components
of $\lfloor B_{Z'}\rfloor$.
We argue that no horizontal divisor is contracted on this minimal model program.
Indeed, if there is a horizontal divisorial contraction, this would induce a divisorial contraction on the general fiber $F_{Z'}$.
This is impossible since such projective variety has Picard rank one.
Analogously, the Mori fiber space on which this MMP terminates
will induce a contraction on the general fiber.
Hence, such contraction must map the general fiber to a point.
Thus, the MFS on which this MMP terminates maps birationally to $T_Z$.
We call such Mori fiber space $Z''\rightarrow T'_Z$.
By construction, the birational map $Z'\dashrightarrow Z''$ 
only contract divisorial log canonical places of $(Z',B_{Z'})$.
By Proposition~\ref{prop:birational-toric}, we conclude that $(Z',B_{Z'})$
is a toric pair provided that $(Z'',B_{Z''})$ is a toric pair.
We proceed to prove this latter statement.

We will prove that the log pair $(Z'',B_{Z''})$ is a log Calabi-Yau toric pair.
The variety $T'_Z$ is a higher toric model of $T_Z$ as 
it only extract log canonical places of the toric pair $(T_Z,B_{T_Z})$.
We denote by $H_1,\dots,H_r$ the torus invariant divisors of $T'_Z$,
where $r\geq c+1$.
Hence, $T'_Z$ has Picard rank $r-c$.
We claim that the number of components
of $\lfloor B_{Z''}\rfloor$ equals $r+f+1$.
$\lfloor B_{Z''}\rfloor$ has $f+1$ horizontal components.
On the other hand, 
$\lfloor B_{Z''}\rfloor$ has a unique prime component
over each divisor $H_i$, being a Mori fiber space.
Thus, $\lfloor B_{Z''}\rfloor$ has $r$ vertical components.
This proves the claim.
Note that $\dim(Z'')=c+f$.
On the other hand, $\rho(Z'')=r-c+1$, given that
$\rho(T_Z)=r-c$ and $\rho(Z''/T_Z)=1$ as it is a Mori fiber space.
Now, we can compute the complexity of the log Calabi-Yau pair $(Z'',B_{Z''})$ as follows:
\[
c(Z'',B_{Z''})=\dim(Z'')+\rho(Z'')-(r+f+1) =0.
\]
By Theorem~\ref{thm:toric-by-complexity}, we conclude that $(Z'',\lfloor B_{Z''}\rfloor)$ is log Calabi-Yau toric pair.
However, $B_{Z''}$ is a reduced divisor.
Hence, $(Z'',B_{Z''})$ is a log Calabi-Yau toric pair.
By Proposition~\ref{prop:birational-toric}, we conclude that $(Z,B_Z)$ is a log Calabi-Yau toric pair as well.\\

\textit{Step 11:} We introduce a projective variety $W$ dominating $X$ and prove that it is toric.\\

We define $W$ to be the normalization of a connected component of the fiber product $X\times_Y Z$.
Note that $W$ is endowed with an action of a finite group $A_W$.
The finite group $A_W$ surjects onto both $A$ and $H$.
We denote by $H_0$ the kernel of the surjection $A_W\rightarrow A$
and by $A_0$ the kernel of the surjection $A_W\rightarrow H$.
There are natural monomorphisms
$A_0<A$ and $H_0<H$.
Since $|H|$ is $O_f(1)$, we conclude that the index of $A_0$ in $A$ is $O_f(1)$.
By Lemma~\ref{lem:bounded-rank-large-k-generation}, 
up to replacing $N$ with an $O_f(1)$ multiple, 
we may assume that $A_0$ contains $\zz^n_N$ as a subgroup.
Furthermore, $W$ has an $A_W$-equivariant fibration to a toric variety $T_W$.
We have a commutative diagram as follows:
\[
\xymatrix{
& W\ar[ld]_-{/H_0}\ar[rd]^-{/A_0}\ar[d] & \\
X^*\ar[d]\ar[rd] & T_W\ar[ld]\ar[rd] & Z\ar[d]\ar[ld] \\
 \pp^c\ar[rd]_-{/A_C} & Y\ar[d] & T_Z\ar[ld]^-{/H} \\
 &  T_Y
}
\]
As usual, we denote by $B_W$ the boundary obtained by the log pull-back
of $K_Z+B_Z$ to $W$.
We claim that $(W,B_W)$ is a projective log Calabi-Yau toric pair.
By construction, the finite morphism
$Z\rightarrow Y$ is unramified over $Y\setminus \supp \lfloor B_Y\rfloor$.
Then, we have that $Y$ is smooth outside the support
of $\lfloor B_Y \rfloor$.
Indeed it is the quotient of a smooth variety by a finite group with trivial stabilizers.
Now, we argue that $X^*\rightarrow Y$ is unramified over
$Y\setminus \supp \lfloor B_Y\rfloor$.
By purity of the branch locus, it suffices to prove that there is no ramification divisor.
If there is a horizontal ramification divisor over $T_Y$, 
then there is a prime divisor of $B_{F_Y}$ which is not reduced, leading to a contradiction.
On the other hand, assume there is a vertical prime ramification divisor over $T_Y$ which maps to the torus.
By Lemma~\ref{lem:vert-div-MFS},
such divisor is the pull-back of a divisor from the base.
Hence, the boundary part of the canonical bundle formula
will intersect the torus of the base.
This leads to a contradiction.
We conclude that $X^*\rightarrow Y$ is unramified on the complement of the log canonical centers of $(Y,B_Y)$.
Thus, the finite morphism $W\rightarrow Z$ is unramified over the torus of $Z$.
By Proposition~\ref{prop:cover-toric-unramified-torus-toric}, we conclude that $(W,B_W)$ is a log Calabi-Yau toric pair.
In particular, we have that the finite quotient $W\rightarrow Z$
is the quotient by a finite subgroup of the torus $\mathbb{G}_m^n$.
Hence, we conclude that $A_0 < \mathbb{G}_m^n \leqslant {\rm Aut}(W,B_W)$.\\

\textit{Step 12:} We conclude that $(X^*,B^*)$ is a log Calabi-Yau toric pair.\\

Note that $H_0$ is a finite automorphism group of the toric pair $(W,B_W)$.
In particular, $H_0$ fixes the reduced toric boundary.
Thus, we have a natural monomorphism
$H_0< {\rm Aut}(\mathbb{G}_m^n)$.
By Lemma~\ref{lem:bounded-rank-large-k-generation}, we know that for $g_n(A)\geq N$ large enough 
the group $A_0$ contains $\zz_{a_0}^n$, 
where $a_0$ is a positive number greater or equal than $N$. By Proposition~\ref{prop:making-H0-toric}, we conclude that for $N$ large enough, only depending on the dimension, the group $H_0$
is actually a subgroup of $\mathbb{G}_m^n\leqslant {\rm Aut}(W,B_W)$.
Thus, by Proposition~\ref{prop:quot-toric-by-toric-finite}, we conclude that $(X^*,B^*)$ is a projective Calabi-Yau toric pair.
Note that toric action of $A_0$ on $W$ 
descends to a finite toric action of $A_0$ on $X^*$.
Furthermore, we have that $A_0<A$ has bounded index only depending on the dimension.
By Proposition~\ref{prop:making-abelian-group-toric}, we conclude that for $g_n(A)\geq N$ large enough compared with the dimension, we have that
$A<\mathbb{G}_m^n\leqslant {\rm Aut}(X^*,B^*)$.
Hence, $A$ acts as a toric finite group on $X^*$.
Thus, it suffices to prove the last assertion of the theorem, which is the final step.\\

\textit{Step 13:} We check that $X^*\dashrightarrow X$ is an isomorphism over the torus.\\

By construction, the contraction morphism $X^*\rightarrow \pp^c$ is a toric morphism.
In particular, the pre-image of the torus of $\pp^c$ contains the torus of $X^*$.
Furthermore, we know that $X^* \dashrightarrow X''$
is an isomorphism over the pre-image of the torus of $\pp^c$.
Hence, we conclude that $X^*\dashrightarrow X''$ is an isomorphism over $\mathbb{G}_m^n\subset X^*$.
The fact that $X^*\dashrightarrow X$ is an isomorphism over the torus is analogous to the proof of the third step. This concludes the proof.
\end{proof}

\begin{theorem}\label{thm:FT-full-rank-non-abelian-G}
Let $n$ be a positive integer and
let $\Lambda\subset \qq$ be a set satisfying the descending chain condition with rational accumulation points.
Then, there exist positive integers
$M:=M(\Lambda,n)$ and $N:=N(n)$,
so that $M$ (resp. $N$) only depends on $\Lambda$ and $n$ 
(resp. $n$),
satisfying the following.
Let $X$ be a Fano type variety of dimension $n$ and $\Delta$ be a boundary on $X$, such that the following conditions hold:
\begin{enumerate}
\item $G<{\rm Aut}(X)$ is a finite subgroup with $g_n(G)\geq M$, 
\item $(X,\Delta)$ is log canonical and $G$-invariant,
\item the coefficients of $\Delta$ belong to $\Lambda$, and
\item $-(K_X+\Delta)$ is $\qq$-complemented.
\end{enumerate}
Then, there exists:
\begin{enumerate}
\item A normal abelian subgroup $A\leqslant G$ of index at most $N$,
\item a boundary $B\geqslant \Delta$ on $X$, and
\item an $A$-equivariant birational map $X\dashrightarrow X'$, 
\end{enumerate}
satisfying the following conditions:
\begin{enumerate}
\item The pair $(X,B)$ is log canonical, $G$-equivariant, and 
$(K_X+B)\sim 0$,
\item the push-forward of $K_X+B$ to $X'$ is a log pair $(X',B')$,
\item the pair $(X',B')$ is a log Calabi-Yau toric pair, and
\item there are group monomorphisms
$A<\mathbb{G}_m^n\leqslant {\rm Aut}(X,B)$.
\end{enumerate}
In particular, $B'$ is the reduced toric boundary of $X'$.
Furthermore, the birational map $X\dashrightarrow X'$ is an isomorphism over $\mathbb{G}_m^n$.
\end{theorem}

\begin{proof}
We construct an abelian normal subgroup of $G$ of bounded index.
By Theorem~\ref{thm:Jordan-FT}, we can find a normal abelian subgroup $A\leqslant G$ of index at most $N:=N(n)$ and rank at most $n$, where $N$ only depends on $n$.
By Proposition~\ref{prop:k-gen-bounded-subgroup}, we know that $g_n(A)\geq g_n(G)/N$.
So, whenever $N$ is fixed and $g_n(G)$ is large, then $g_n(A)$ is large as well.
Now, we proceed to construct a bounded $G$-equivariant complement for $(X,\Delta)$.
Note that the log pair $(X,\Delta)$ satisfies all the hypotheses of Theorem~\ref{thm:G-equiv-complement}.
Hence, there exists a $G$-equivariant log canonical $M_0$-complement $B\geq \Delta\geq 0$.
The pair $(X,B)$ is log canonical, $G$-equivariant, and $M_0(K_X+B)\sim 0$,
where $M_0$ only depend on the natural number $n$ and the set $\Lambda$.
Observe that $(X,B)$ is also an $A$-equivariant complement.
Then we conclude by Theorem~\ref{thm:FT-full-rank}..
\end{proof}

\subsection{Proof of the main local theorem} 
In this subsection, we prove the main local theorem of this article.
We start by stating a more general version of Theorem~\ref{introthm:degeneration}.

\begin{theorem}\label{thm:regional-general}
Let $n$ be a positive integer
and let $\Lambda$ be a set of rational numbers satisfying the descending chain condition with rational accumulation points.
There exists a positive integer $N:=N(n,\Lambda)$,
only depending on $n$ and $\Lambda$,
satisfying the following.
Let $x\in (X,\Delta)$ be a $n$-dimensional klt singularity
so that the coefficients of $\Delta$ belong to $\Lambda$.
Assume that 
$g_n(\pi_1^{\rm reg}(X,\Delta;x))\geq N$.
Then, $x\in (X,\Delta)$ degenerate to a lce-tq singularity.
\end{theorem}

\begin{proof}
Let $(X,B)$ be a $M$-complement of $(X,\Delta)$ around $x\in X$.
By Proposition~\ref{prop:existence-universal-cover}, we can take
the universal cover $\pi\colon Y\rightarrow X$ of the regional 
fundamental group of $(X,\Delta)$ at $x$.
Let $X'\rightarrow X$ be a plt blow-up of $(X,\Delta)$ at $x\in X$.
We may assume that the exceptional divisor of the plt blow-up
is a log canonical center of $(X,B)$.
Denote by $E$ the exceptional divisor of the plt blow-up.
Note that $X$ is the quotient of $Y$ by $G\simeq \pi_1^{\rm reg}(X,\Delta;x)$ and $Y$ contains a unique fixed point $y\in Y$ over $x\in X$.
Let $Y'$ be the normalization of the main component of $Y\times_X X'$.
Note that $Y'\rightarrow Y$ is a $G$-invariant plt blow-up with exceptional divisor $E_Y$
(see the proof of Lemma~\ref{lem:existence-g-inv-plt}).
We know that $E_Y$ is $G$-invariant because it maps to the $G$-fixed point $y\in Y$.
Let $C$ be the cyclic normal subgroup of $G$ acting on $E_Y$ as the identity.
Let $H$ be the quotient of $G$ by $C$.
By Lemma~\ref{lem:quot-cycl-k-gen}, we have that $g_{n-1}(H) \geq g_n(G)$.
Let $(Y,B_Y)$ be the log pull-back of $(X,B)$ to $Y$, let $(Y',B_{Y'})$ be its log pull-back to $Y'$.
Let $(E_Y,B_{E_Y})$ be the pair obtained by adjunction
$K_{Y'}+B_{Y'}|_{E_Y}$.
Note that the coefficient of $B_{E_Y}$ belong to a finite set
which only depends on $M$.
We can apply Theorem~\ref{introthm:main-thm},
to the $H$-invariant pair $(E_Y,B_{E_Y})$. 
There exists a constant $N(n-1)$,
only depending on $n-1$, so that $(E_Y,B_{E_Y})$ is crepant equivalent  toric
provided that $g_{n-1}(H)\geq N(n-1)$.
Hence, it suffices to take $g_n(G)\geq N(n-1)$
to assume that $E_Y$ is log crepant equivalent  toric.
We conclude that $E$ is the quotient of a log crepant equivalent  toric variety by a finite group.
Hence, $E$ is a log crepant equivalent  toric quotient projective variety.
Using cone degenerations associated to plt blow-ups (see, e.g.,~\cite{HLM20}*{2.5}), 
we conclude that the singularity $x\in (X,\Delta)$ degenerates
to a log crepant equivalent  toric quotient singularity.
\end{proof}

\section{Applications}\label{sec:app}
In this section, we prove some applications of the main projective and local statements.
Before proceeding to the applications, we give an example of a lce-tq singularity which is not the quotient of a toric singularity.
This example is due to Constantin Shramov.

\begin{example}\label{example}
{\em 
Consider the weighted hypersurface
$X_n=\{x^{2n}+y^{2n}+zw=0\} \subset
\pp(1,1,n,n)$.
Note that we have 
$A_n:=\zz_{2n}^2< {\rm Aut}(X_n)$.
The first generator is given by 
$x\mapsto \eta x$, where $\eta$ is a $2n$-root of unity.
The second generator is given by
$(z,w)\mapsto (\eta z, \eta^{-1}w)$.
$X_n$ is a Fano type surface.
However, $X_n$ is not toric.
Hence, by Theorem~\ref{introthm:main-thm}, we know that $X_n$ is log crepant equivalent  toric.
We give an explicit description of the complement in this case.
The curves $C_x:=\{x=0\}$
and $C_y:=\{y=0\}$ are irreducible.
The curves $C_w:=\{w=0\}$ and
$C_z:=\{z=0\}$ have $2n$ irreducible components each.
The pair $(X_n,C_x+C_y)$ is log canonical
and $K_{X_n}+C_x+C_y\sim 0$.
Hence, the points $[0:0:1:0]$
and $[0:0:0:1]$ are toroidal points of $X_n$.
We can blow-up both of them and obtain a surface $X'_n$ where the $2n$ irreducible components of $C_w$ and $C_z$ are separated.
The log pull-back of $K_{X_n}+C_x+C_y$ to $X'_n$ equals
$K_{X'_n}+C'_x+C'_y+E_1+E_2$.
Then, we can run a $A_n$-equivariant minimal model program for $K_{X'_n}$.
This MMP contracts the strict transform of 
$C_w$ and $C_z$.
It terminates
with an $A_n$-equivariant Mori fiber space structure
$\pp^1\times\pp^1 \rightarrow \pp^1$.
Here, $A_n$ acts as the multiplication by roots of unity on the torus of $\pp^1\times\pp^1$.
Furthermore, the strict transform of $C'_x+C'_y+E_1+E_2$ is the toric boundary
of $\pp^1\times \pp^1$.
The involutions $(x,y)\mapsto (y,x)$ and
$(z,w)\mapsto (w,z)$ are induced by the corresponding involutions swapping the two components of each copy of $\pp^1$ on $\pp^1\times\pp^1$.
Then, the cone over $X_n$ with respect to $-K_{X_n}$ is a $3$-fold lce-tq singularity 
$y_n\in Y_n$. It is not a toric quotient singularity. Indeed, $X_n$ is not the quotient of a toric singularity.
}
\end{example}

Now, we turn to prove Theorem~\ref{introthm:complements}.
The following is a version of such theorem
that considers more general coefficient sets.

\begin{theorem}
Let $n$ be a positive integer and let $\Lambda$ be a set of rational numbers satisfying the descending chain condition
with rational accumulation points.
There exists a positive integer $N:=N(n,\Lambda)$, only depending on $n$ and $\Lambda$, satisfying the following.
Let $x\in (X,\Delta)$ be a $n$-dimensional klt singularity so that the coefficients of $\Delta$ belong to $\Lambda$.
Assume that $g_n(\pi_1^{\rm reg}(X,\Delta;x))\geq N$.
Then, there exists a cover $\pi\colon Z\rightarrow X$ of 
$(X,\Delta)$ of degree at most $N(n)$, so that the log pull-back $(Z,\Delta_Z)$
of $(X,\Delta)$ to $Z$
admits a $1$-complement.
Furthermore, the klt pair $(Z,\Delta_Z)$ degenerates to a 
lce-tq singularity which contains an open affine isomorphic to an algebraic torus $\mathbb{G}_m^n$.
\end{theorem}

\begin{proof}
We use the notation of the proof of Theorem~\ref{thm:regional-general}.
Note that $(E_Y,B_{E_Y})$ is crepant equivalent  to a log Calabi-Yau 
toric pair $(T,B_T)$.
The group $H$ admits an abelian normal subgroup $A_H\leqslant H$
for which the birational map $T\dashrightarrow E_Y$ is $A_H$-equivariant and $A_H<\mathbb{G}_m^{n-1}\leqslant {\rm Aut}(T,B_T)$.
Furthermore, the index of $A_H$ in $H$ is bounded by a constant
which only depends on the dimension $n-1$.
Hence, the quotient of $(E_Y,B_{E_Y})$ by $A_H$
is crepant equivalent  to a log Calabi-Yau toric pair.
Indeed, it is crepant
equivalent to the quotient
of $(T,B_T)$ by $A_H$.
Let $(E_{Z},B_{E_{Z}})$ be the quotient of $(E_Y,B_{E_Y})$ by $A_H$.
By construction, we have that
$K_{E_Z}+B_{E_Z}\sim 0$ and
$B_{E_Z}$ is a reduced boundary.
We denote by $A$ the pre-image of $A_H$ in $G$.
Note that $A$ is a normal subgroup of $G$.
We denote $N_H:=H/A_H$ and $N:=G/A$.
We have a commutative diagram as follows:
\[ 
\xymatrix{
(E_Y,B_{E_Y}) \ar[r]^-{/A_H}\ar@{^{(}->}[d] & (E_Z,B_{E_Z}) \ar@{^{(}->}[d] \ar[r]^-{/N_H} & (E,B_E) \ar@{^{(}->}[d] \\
Y'\ar[d]\ar[r]^-{/A} & Z' \ar[d]\ar[r]^-{/N} & X'\ar[d] \\
Y\ar[r]^-{/A} & Z \ar[r]^-{/N} & X
}
\]
where all the horizontal morphisms are quotients, the lower vertical morphisms are equivariant plt blow-ups, 
and the upper vertical maps are equivariant embeddings of
the exceptional divisor of the plt blow-up.
Note that $B_{E_Z}$ is a $1$-complement for $K_{E_Z}$.
Indeed, the pair
$(E_Z,B_{E_Z})$ is log crepant equivalent to a toric pair.
We denote by $z\in Z$ the image of $y$ on $Z$.
We can lift the complement $B_{E_{Z}}$ of $K_{E_{Z}}$ to a complement 
$B_{Z'}$ of 
$K_{Z'}$,
over a neighborhood of $z\in Z$.
This is also a $1$-complement.
In particular, we have that
$(Z',B_{Z'})$ is log canonical and
$K_{Z'}+B_{Z'}\sim 0$
over a neighborhood of $z\in Z$.
We let $K_Z+B_Z$ to be the push-forward of
$K_{Z'}+B_{Z'}$ to $Z$.
Hence, $K_Z+B_Z$ is a $1$-complement of $K_Z$.
Denote $\pi_Z\colon Z\rightarrow X$.
Hence, we have that
$K_Z+B_Z\geq \pi_Z^*(K_X+\Delta)=K_Z+\Delta_Z$.
We conclude that $(Z,B_Z)$
is a $1$-complement
of the log pull-back 
$(Z,\Delta_Z)$ of $(X,\Delta)$ to $Z$.
By construction,
the finite morphism
$Z\rightarrow X$ has degree at most $N$.

Now, we proceed to study the cone degeneration of $Z$ with respect to the plt blow-up $Z'\rightarrow Z$.
Note that the torus of $E_Z$ is disjoint from the support of $B_{E_Z}$.
Let $K_Z+\Delta_Z$ be the log pull-back of $K_X+\Delta$ to $Z$.
Let $K_Z'+\Delta_Z'$ be the log pull-back of $K_Z+\Delta_Z$ to $Z'$ after increasing the coefficient at $E_Z$ to one.
Let $K_{E_Z}+\Delta_{E_Z}$ be the pair obtained by adjunction of 
$K_Z'+\Delta_Z'$ to $E_Z$. 
By construction, we have that $\Delta_{Z'}\leq B_{E_Z}$.
Then, the fractional coefficients of $\Delta_{Z'}$ are disjoint from the torus.
We conclude that $z\in (Z,\Delta_Z)$ degenerate to an orbifold cone singularity over $E_Z$ so that the corresponding ample $\qq$-divisor is Weil on the torus of $E_Z$.
In particular, this ample
$\qq$-divisor is Cartier on the torus of $E_Z$.
Hence, such cone contains an algebraic torus as an open subvariety.
\end{proof}

\begin{proof}[Proof of Theorem~\ref{introthm:log-smooth-locus}]
Let $(X,B)$ be a $M$-complement of $X$.
Let $Y^{\rm sm}\rightarrow X^{\rm sm}$ 
be the universal cover of $G:=\pi_1(X^{\rm sm})$.
Let $Y$ be the normalization of $X$ on the function field of $Y^{\rm sm}$.
Then, we have a $G$-equivariant morphism $Y\rightarrow X$ so
that $Y$ is also of Fano type.
Indeed, by construction the finite morphism
$Y\rightarrow X$ is unramified in codimension one.
Let $(Y,B_Y)$ be the pull-back of $(X,B)$ to $Y$.
By Theorem~\ref{introthm:main-thm}, we know that $(Y,B_Y)$ is crepant equivalent  toric whenever $k\geq N(n)$,
for some constant $N(n)$, which only depends on $n$.
Hence, we conclude that $X$ is a log crepant equivalent  toric quotient.
\end{proof}

\subsection{Dual complexes} In this subsection, we recall the dual complex of a log canonical pair and prove a result regarding dual complexes of log crepant equivalent  toric quotients.

\begin{definition}\label{def:dual-complex}{\em 
Let $E$ be a simple normal crossing variety
with irreducible components $E_1,\dots,E_k$.
An {\em stratum} of $E$ is an irreducible component of an intersection of some of the $E_i$'s.
The {\em dual complex} of $E$,
denoted by $\mathcal{D}(E)$,
is a CW-complex whose vertices are labeled by the irreducible components of $E$,
and we attach at $(|J|-1)$-cell for every stratum of $\cap_{i\in J}E_i$.
}
\end{definition}

\begin{definition}{\em 
Let $(X,B)$ be a projective log Calabi-Yau pair.
We define the {\em dual complex} of $(X,B)$,
denoted by $\mathcal{D}(X,B)$, to be the dual complex
of $\lfloor B_Y\rfloor$ where $(Y,B_Y)$ is a
$\qq$-factorial dlt modification of $(X,B)$.
}
\end{definition}

Recall from~\cite{dFKX17}, that the dual complex of projective, 
log canonical, crepant equivalent  birational pairs are PL-homeomorphic to each other. Hence, the dual complex of a log Calabi-Yau pair is well-defined.
We introduce a version of Theorem~\ref{introthm:dual-complex} which allows more general coefficient sets.
It is clear that 
Theorem~\ref{thm:dual-complex-general-cOeff} below implies 
Theorem~\ref{introthm:dual-complex}.

\begin{theorem}\label{thm:dual-complex-general-cOeff}
Let $n$ be a positive integer and let $\Lambda$ be a set of rational numbers satisfying the descending chain condition
with rational accumulation points.
Then, there exists a positive integer $N:=N(n,\Lambda)$, only depending on $n$ and $\Lambda$, satisfying the following.
Let $X$ be a Fano type variety and
$\Delta$ be an effective divisor on $X$ so that:
\begin{enumerate}
\item the coefficients of $\Delta$ belong to $\Lambda$, and
\item $-(K_X+\Delta)$ is $\qq$-complemented.
\end{enumerate}
Assume that
$g_n(\pi_1^{\rm alg}(X,\Delta))\geq N$.
Then, there exists a boundary $B$ on $X$
so that
$\mathcal{D}(X,B)\simeq_{\rm PL} S^{n-1}/G$
where $G$ is a finite group with $|G|\leq N$.
\end{theorem}

\begin{proof}[Proof of Theorem~\ref{thm:dual-complex-general-cOeff}]
Note that the pair 
$(X,\Delta)$ satisfies all the conditions of Theorem~\ref{thm:G-equiv-complement}.
Let $(X,B)$ be a $M$-complement of $X$.
Here, $M$ only depends on $n$ and $\Lambda$.
Let $Y^{\rm sm}\rightarrow X^{\rm sm}$ 
be the universal cover of $\pi_1^{\rm alg}(X,\Delta)=:G$.
By assumption,
we have that
$g_n(G)\geq N$.
Let $Y$ be the normalization of $X$ on the function field of $Y^{\rm sm}$.
Then, we have a $G$-equivariant morphism $Y\rightarrow X$ so
that $Y$ is also of Fano type.
Let $(Y,B_Y)$ be the pull-back of $(X,B)$ to $Y$.
By Theorem~\ref{introthm:main-thm}, we know that $(Y,B_Y)$ is crepant equivalent  toric whenever $N \geq N_0(n)$,
for some constant $N_0(n)$, which only depends on $n$.
By~\cite{dFKX17}, we know that $\mathcal{D}(Y,B_Y)\simeq S^{n-1}$.
Furthermore, we know that $G$ admits an abelian subgroup $A\leqslant G$ so that $A$ acts as the multiplication by roots of unity on the toric model.
Let $(Z,B_Z)$ be the quotient of $(Y,B_Y)$ by $A$.
Then, applying~\cite{dFKX17} again, we have that
$\mathcal{D}(Z,B_Z)\simeq S^{n-1}$.
Note that $(X,B)$ is the quotient by $(Z,B_Z)$ by a group
$N:=G/A$
with order bounded by $N_0(n)$.
Thus, we conclude that the dual complex $\mathcal{D}(X,B)$
is the quotient of $S^{n-1}$ by a group whose order is bounded by a function of the dimension $n$.
\end{proof}


\begin{bibdiv}
\begin{biblist}

\bib{Amb05}{article}{
   author={Ambro, Florin},
   title={The moduli $b$-divisor of an lc-trivial fibration},
   journal={Compos. Math.},
   volume={141},
   date={2005},
   number={2},
   pages={385--403},
   issn={0010-437X},
   review={\MR{2134273}},
   doi={10.1112/S0010437X04001071},
}

\bib{AW97}{article}{
   author={Abramovich, Dan},
   author={Wang, Jianhua},
   title={Equivariant resolution of singularities in characteristic $0$},
   journal={Math. Res. Lett.},
   volume={4},
   date={1997},
   number={2-3},
   pages={427--433},
   issn={1073-2780},
   review={\MR{1453072}},
   doi={10.4310/MRL.1997.v4.n3.a11},
}
	
\bib{Amb06}{article}{
   author={Ambro, Florin},
   title={The set of toric minimal log discrepancies},
   journal={Cent. Eur. J. Math.},
   volume={4},
   date={2006},
   number={3},
   pages={358--370},
   issn={1895-1074},
   review={\MR{2233855}},
   doi={10.2478/s11533-006-0013-x},
}

\bib{BB92}{article}{
 Author = {Borisov, A. A.},
 Author = {Borisov, L. A.},
 Title = {Singular toric Fano varieties.},
 Journal = {Russ. Acad. Sci., Sb., Math.},
 ISSN = {1064-5616},
 Volume = {75},
 Number = {1},
 Pages = {277--283},
 Year = {1992},
 Language = {English},
}

\bib{BCHM10}{article}{
   author={Birkar, Caucher},
   author={Cascini, Paolo},
   author={Hacon, Christopher D.},
   author={McKernan, James},
   title={Existence of minimal models for varieties of log general type},
   journal={J. Amer. Math. Soc.},
   volume={23},
   date={2010},
   number={2},
   pages={405--468},
   issn={0894-0347},
   review={\MR{2601039}},
   doi={10.1090/S0894-0347-09-00649-3},
}

\bib{BFMS20}{article}{
  author = {Braun, Lukas},
  author = {Filipazzi, Stefano},
  author = {Moraga, Joaqu\'in},
  author = {Svaldi, Roberto},
  title={The Jordan property for local fundamental groups},
  eprint={arXiv:2006.01253},
  url={https://arxiv.org/abs/2006.01253},
  date = {2020},
}

\bib{Bir19}{article}{
   author={Birkar, Caucher},
   title={Anti-pluricanonical systems on Fano varieties},
   journal={Ann. of Math. (2)},
   volume={190},
   date={2019},
   number={2},
   pages={345--463},
   issn={0003-486X},
   review={\MR{3997127}},
   doi={10.4007/annals.2019.190.2.1},
}
	
\bib{Bir21}{article}{
   author={Birkar, Caucher},
   title={Singularities of linear systems and boundedness of Fano varieties},
   journal={Ann. of Math. (2)},
   volume={193},
   date={2021},
   number={2},
   pages={347--405},
   issn={0003-486X},
   review={\MR{4224714}},
   doi={10.4007/annals.2021.193.2.1},
}	

\bib{Bor91}{book}{
   author={Borel, Armand},
   title={Linear algebraic groups},
   series={Graduate Texts in Mathematics},
   volume={126},
   edition={2},
   publisher={Springer-Verlag, New York},
   date={1991},
   pages={xii+288},
   isbn={0-387-97370-2},
   review={\MR{1102012}},
   doi={10.1007/978-1-4612-0941-6},
}

\bib{BZ16}{article}{
   author={Birkar, Caucher},
   author={Zhang, De-Qi},
   title={Effectivity of Iitaka fibrations and pluricanonical systems of
   polarized pairs},
   journal={Publ. Math. Inst. Hautes \'Etudes Sci.},
   volume={123},
   date={2016},
   pages={283--331},
   issn={0073-8301},
   review={\MR{3502099}},
}

\bib{Bra20}{article}{
  author = {Braun, Lukas},
  title={The local fundamental group of a Kawamata log terminal singularity is finite},
  eprint = {arXiv:2004.00522},
  url={https://arxiv.org/abs/2004.00522},
  date = {2020},
}

\bib{BMSZ18}{article}{
   author={Brown, Morgan V.},
   author={McKernan, James},
   author={Svaldi, Roberto},
   author={Zong, Hong R.},
   title={A geometric characterization of toric varieties},
   journal={Duke Math. J.},
   volume={167},
   date={2018},
   number={5},
   pages={923--968},
   issn={0012-7094},
   review={\MR{3782064}},
   doi={10.1215/00127094-2017-0047},
}

\bib{CLS11}{book}{
   author={Cox, David A.},
   author={Little, John B.},
   author={Schenck, Henry K.},
   title={Toric varieties},
   series={Graduate Studies in Mathematics},
   volume={124},
   publisher={American Mathematical Society, Providence, RI},
   date={2011},
   pages={xxiv+841},
   isbn={978-0-8218-4819-7},
   review={\MR{2810322}},
   doi={10.1090/gsm/124},
}

\bib{Cox95}{article}{
   author={Cox, David A.},
   title={The homogeneous coordinate ring of a toric variety},
   journal={J. Algebraic Geom.},
   volume={4},
   date={1995},
   number={1},
   pages={17--50},
   issn={1056-3911},
   review={\MR{1299003}},
}
	
\bib{dFdC16}{article}{
   author={de Fernex, Tommaso},
   author={Docampo, Roi},
   title={Terminal valuations and the Nash problem},
   journal={Invent. Math.},
   volume={203},
   date={2016},
   number={1},
   pages={303--331},
   issn={0020-9910},
   review={\MR{3437873}},
   doi={10.1007/s00222-015-0597-5},
}

\bib{dFKX17}{article}{
   author={de Fernex, Tommaso},
   author={Koll\'{a}r, J\'{a}nos},
   author={Xu, Chenyang},
   title={The dual complex of singularities},
   conference={
      title={Higher dimensional algebraic geometry---in honour of Professor
      Yujiro Kawamata's sixtieth birthday},
   },
   book={
      series={Adv. Stud. Pure Math.},
      volume={74},
      publisher={Math. Soc. Japan, Tokyo},
   },
   date={2017},
   pages={103--129},
   review={\MR{3791210}},
   doi={10.2969/aspm/07410103},
}
		

\bib{FG14}{article}{
   author={Fujino, Osamu},
   author={Gongyo, Yoshinori},
   title={On the moduli b-divisors of lc-trivial fibrations},
   language={English, with English and French summaries},
   journal={Ann. Inst. Fourier (Grenoble)},
   volume={64},
   date={2014},
   number={4},
   pages={1721--1735},
   issn={0373-0956},
   review={\MR{3329677}},
}

\bib{FM20}{article}{
   author={Filipazzi, Stefano},
   author={Moraga, Joaqu\'{\i}n},
   title={Strong $(\delta,n)$-complements for semi-stable morphisms},
   journal={Doc. Math.},
   volume={25},
   date={2020},
   pages={1953--1996},
   issn={1431-0635},
   review={\MR{4187715}},
}

\bib{Ful93}{book}{
   author={Fulton, William},
   title={Introduction to toric varieties},
   series={Annals of Mathematics Studies},
   volume={131},
   note={The William H. Roever Lectures in Geometry},
   publisher={Princeton University Press, Princeton, NJ},
   date={1993},
   pages={xii+157},
   isbn={0-691-00049-2},
   review={\MR{1234037}},
   doi={10.1515/9781400882526},
}

\bib{HLM20}{article}{
   author={Han, Jingjun},
   author={Liu, Jihao},
   author={Moraga, Joaqu\'{\i}n},
   title={Bounded deformations of $(\epsilon,\delta)$-log canonical
   singularities},
   journal={J. Math. Sci. Univ. Tokyo},
   volume={27},
   date={2020},
   number={1},
   pages={1--28},
   issn={1340-5705},
   review={\MR{4246623}},
}

\bib{HLS19}{article}{
   author={Han, Jingjun},
   author={Liu, Jihao},
   author={Shokurov, Vyacheslav V.},
   title={ACC for minimal log discrepancies of exceptional singularities},
   eprint={arXiv:1903.04338},
   url={https://arxiv.org/abs/1903.04338},
   date={2019},
}

\bib{HK00}{article}{
   author={Hu, Yi},
   author={Keel, Sean},
   title={Mori dream spaces and GIT},
   note={Dedicated to William Fulton on the occasion of his 60th birthday},
   journal={Michigan Math. J.},
   volume={48},
   date={2000},
   pages={331--348},
   issn={0026-2285},
   review={\MR{1786494}},
   doi={10.1307/mmj/1030132722},
}

\bib{Jor73}{article}{
   author={Jordan, Camille},
   title={M\'{e}moire sur une application de la th\'{e}orie des substitutions \`a
   l'\'{e}tude des \'{e}quations diff\'{e}rentielles lin\'{e}aires},
   language={French},
   journal={Bull. Soc. Math. France},
   volume={2},
   date={1873/74},
   pages={100--127},
   issn={0037-9484},
   review={\MR{1503686}},
}

\bib{KM98}{book}{
   author={Koll\'{a}r, J\'{a}nos},
   author={Mori, Shigefumi},
   title={Birational geometry of algebraic varieties},
   series={Cambridge Tracts in Mathematics},
   volume={134},
   note={With the collaboration of C. H. Clemens and A. Corti;
   Translated from the 1998 Japanese original},
   publisher={Cambridge University Press, Cambridge},
   date={1998},
   pages={viii+254},
   isbn={0-521-63277-3},
   review={\MR{1658959}},
   doi={10.1017/CBO9780511662560},
}

\bib{Kol92}{book}{ 
Editor = {Koll\'ar, J., and others},
TITLE = {Flips and abundance for algebraic threefolds},
      NOTE = {Papers from the Second Summer Seminar on Algebraic Geometry
              held at the University of Utah, Salt Lake City, Utah, August
              1991,
              Ast\'{e}risque No. 211 (1992) (1992)},
 PUBLISHER = {Soci\'{e}t\'{e} Math\'{e}matique de France, Paris},
      YEAR = {1992},
     PAGES = {1--258},
      ISSN = {0303-1179},
}

\bib{KX16}{article}{
   author={Koll\'{a}r, J\'{a}nos},
   author={Xu, Chenyang},
   title={The dual complex of Calabi-Yau pairs},
   journal={Invent. Math.},
   volume={205},
   date={2016},
   number={3},
   pages={527--557},
   issn={0020-9910},
   review={\MR{3539921}},
   doi={10.1007/s00222-015-0640-6},
}
	
	
\bib{Lai11}{article}{
   author={Lai, Ching-Jui},
   title={Varieties fibered by good minimal models},
   journal={Math. Ann.},
   volume={350},
   date={2011},
   number={3},
   pages={533--547},
   issn={0025-5831},
   review={\MR{2805635}},
   doi={10.1007/s00208-010-0574-7},
}

\bib{Mor18a}{article}{
  author = {Moraga, Joaqu\'in},
  title={On minimal log discrepancies and Koll\'ar components},
  eprint={arXiv:1810:10137},
  url={https://.org/abs/1810.10137},
  date = {2018},
}

\bib{Mor18b}{article}{
  author = {Moraga, Joaqu\'in},
  title={A boundedness theorem for cone singularities},
  year = {2018},
  eprint={arXiv:1812.04670},
  url={https://arxiv.org/abs/1812.04670},
  date={2018},
}

\bib{Mor19}{article}{
author = {Moraga, Joaqu\'in},
title={Extracting non-canonical places},
eprint={arXiv:1911.0991},
url={https://arxiv.org/abs/1911.00991},
date= {2019},
}

\bib{Mor20}{article}{
  author = {Moraga, Joaqu\'in},
  title={Fano type surfaces with large cyclic automorphisms},
  eprint={arXiv:2001.03797},
  url = {https://arxiv.org/abs/2001.03797},
  date = {2020},
}

\bib{PS14}{article}{
   author={Prokhorov, Yuri},
   author={Shramov, Constantin},
   title={Jordan property for groups of birational selfmaps},
   journal={Compos. Math.},
   volume={150},
   date={2014},
   number={12},
   pages={2054--2072},
   issn={0010-437X},
   review={\MR{3292293}},
   doi={10.1112/S0010437X14007581},
}

\bib{PS16}{article}{
   author={Prokhorov, Yuri},
   author={Shramov, Constantin},
   title={Jordan property for Cremona groups},
   journal={Amer. J. Math.},
   volume={138},
   date={2016},
   number={2},
   pages={403--418},
   issn={0002-9327},
   review={\MR{3483470}},
   doi={10.1353/ajm.2016.0017},
}

\bib{Pro00}{article}{
   author={Prokhorov, Yu. G.},
   title={Blow-ups of canonical singularities},
   conference={
      title={Algebra},
      address={Moscow},
      date={1998},
   },
   book={
      publisher={de Gruyter, Berlin},
   },
   date={2000},
   pages={301--317},
   review={\MR{1754677}},
}

\bib{Sho93}{article}{
   author={Shokurov, V. V.},
   title={Three-dimensional log perestroikas},
   language={Russian},
   journal={Izv. Ross. Akad. Nauk Ser. Mat.},
   volume={56},
   date={1992},
   number={1},
   pages={105--203},
   issn={1607-0046},
   translation={
      journal={Russian Acad. Sci. Izv. Math.},
      volume={40},
      date={1993},
      number={1},
      pages={95--202},
      issn={1064-5632},
   },
   review={\MR{1162635}},
   doi={10.1070/IM1993v040n01ABEH001862},
}

\bib{Sho00}{article}{
   author={Shokurov, V. V.},
   title={Complements on surfaces},
   note={Algebraic geometry, 10},
   journal={J. Math. Sci. (New York)},
   volume={102},
   date={2000},
   number={2},
   pages={3876--3932},
   issn={1072-3374},
   review={\MR{1794169}},
   doi={10.1007/BF02984106},
}	

\bib{Tsu83}{article}{
   author={Tsunoda, Shuichiro},
   title={Structure of open algebraic surfaces. I},
   journal={J. Math. Kyoto Univ.},
   volume={23},
   date={1983},
   number={1},
   pages={95--125},
   issn={0023-608X},
   review={\MR{692732}},
   doi={10.1215/kjm/1250521613},
}

\bib{Xu14}{article}{
   author={Xu, Chenyang},
   title={Finiteness of algebraic fundamental groups},
   journal={Compos. Math.},
   volume={150},
   date={2014},
   number={3},
   pages={409--414},
   issn={0010-437X},
   review={\MR{3187625}},
   doi={10.1112/S0010437X13007562},
}

\end{biblist}
\end{bibdiv}
\end{document}